\documentclass[oneside,english]{amsart}
\usepackage[T1]{fontenc}
\usepackage[latin9]{inputenc}
\usepackage{babel}
\usepackage{amstext}
\usepackage{amsthm}
\usepackage{amssymb}
\usepackage[pdfusetitle,
 bookmarks=true,bookmarksnumbered=false,bookmarksopen=false,
 breaklinks=false,pdfborder={0 0 1},backref=section,colorlinks=false]
 {hyperref}

\makeatletter
\numberwithin{equation}{section}
\numberwithin{figure}{section}
\theoremstyle{plain}
\newtheorem{thm}{\protect\theoremname}[section]
\theoremstyle{definition}
\newtheorem{defn}[thm]{\protect\definitionname}
\theoremstyle{plain}
\newtheorem{fact}[thm]{\protect\factname}
\theoremstyle{remark}
\newtheorem{rem}[thm]{\protect\remarkname}
\theoremstyle{plain}
\newtheorem{prop}[thm]{\protect\propositionname}
\theoremstyle{remark}
\newtheorem{claim}[thm]{\protect\claimname}
\theoremstyle{lemma}
\newtheorem{question}[thm]{\protect\questionname}
\theoremstyle{definition}
\newtheorem{example}[thm]{\protect\examplename}
\theoremstyle{definition}
\newtheorem{problem}[thm]{\protect\problemname}
\theoremstyle{plain}
\newtheorem{cor}[thm]{\protect\corollaryname}
\theoremstyle{plain}
\newtheorem{lem}[thm]{\protect\lemmaname}

\usepackage{amsfonts}
\usepackage{nicefrac}
\usepackage{amscd}
\usepackage{a4wide}
\usepackage{url}

\AtBeginDocument{
  
}

\makeatother

\providecommand{\claimname}{Claim}
\providecommand{\corollaryname}{Corollary}
\providecommand{\definitionname}{Definition}
\providecommand{\examplename}{Example}
\providecommand{\factname}{Fact}
\providecommand{\lemmaname}{Lemma}
\providecommand{\problemname}{Problem}
\providecommand{\propositionname}{Proposition}
\providecommand{\questionname}{Question}
\providecommand{\remarkname}{Remark}
\providecommand{\theoremname}{Theorem}

\begin{document}
\global\long\def\p{\mathbf{p}}%
\global\long\def\q{\mathbf{q}}%
\global\long\def\C{\mathfrak{C}}%
\global\long\def\SS{\mathcal{P}}%
 
\global\long\def\pr{\operatorname{pr}}%
\global\long\def\image{\operatorname{im}}%
\global\long\def\otp{\operatorname{otp}}%
\global\long\def\dec{\operatorname{dec}}%
\global\long\def\suc{\operatorname{suc}}%
\global\long\def\pre{\operatorname{pre}}%
\global\long\def\qe{\operatorname{qf}}%
 
\global\long\def\ind{\operatorname{ind}}%
\global\long\def\Nind{\operatorname{Nind}}%
\global\long\def\lev{\operatorname{lev}}%
\global\long\def\Suc{\operatorname{Suc}}%
\global\long\def\HNind{\operatorname{HNind}}%
\global\long\def\minb{{\lim}}%
\global\long\def\concat{\frown}%
\global\long\def\cl{\operatorname{cl}}%
\global\long\def\tp{\operatorname{tp}}%
\global\long\def\id{\operatorname{id}}%
\global\long\def\cons{\left(\star\right)}%
\global\long\def\qf{\operatorname{qf}}%
\global\long\def\ai{\operatorname{ai}}%
\global\long\def\dtp{\operatorname{dtp}}%
\global\long\def\acl{\operatorname{acl}}%
\global\long\def\nb{\operatorname{nb}}%
\global\long\def\limb{{\lim}}%
\global\long\def\leftexp#1#2{{\vphantom{#2}}^{#1}{#2}}%
\global\long\def\intr{\operatorname{interval}}%
\global\long\def\atom{\emph{at}}%
\global\long\def\I{\mathfrak{I}}%
\global\long\def\uf{\operatorname{uf}}%
\global\long\def\ded{\operatorname{ded}}%
\global\long\def\Ded{\operatorname{Ded}}%
\global\long\def\Df{\operatorname{Df}}%
\global\long\def\Th{\operatorname{Th}}%
\global\long\def\eq{\operatorname{eq}}%
\global\long\def\Aut{\operatorname{Aut}}%
\global\long\def\ac{ac}%
\global\long\def\DfOne{\operatorname{df}_{\operatorname{iso}}}%
\global\long\def\modp#1{\pmod{#1}}%
\global\long\def\sequence#1#2{\left\langle #1\,\middle|\,#2\right\rangle }%
\global\long\def\set#1#2{\left\{  #1\,\middle|\,#2\right\}  }%
\global\long\def\Diag{\operatorname{Diag}}%
\global\long\def\Nn{\mathbb{N}}%
\global\long\def\mathrela#1{\mathrel{#1}}%
\global\long\def\twiddle{\mathord{\sim}}%
\global\long\def\mathordi#1{\mathord{#1}}%
\global\long\def\Qq{\mathbb{Q}}%
\global\long\def\dense{\operatorname{dense}}%
\global\long\def\Rr{\mathbb{R}}%
 
\global\long\def\cof{\operatorname{cof}}%
\global\long\def\tr{\operatorname{tr}}%
\global\long\def\treeexp#1#2{#1^{\left\langle #2\right\rangle _{\tr}}}%
\global\long\def\x{\times}%
\global\long\def\forces{\Vdash}%
\global\long\def\Vv{\mathbb{V}}%
\global\long\def\Uu{\mathbb{U}}%
\global\long\def\tauname{\dot{\tau}}%
\global\long\def\ScottPsi{\Psi}%
\global\long\def\cont{2^{\aleph_{0}}}%
\global\long\def\MA#1{{MA}_{#1}}%
\global\long\def\rank#1#2{R_{#1}\left(#2\right)}%
\global\long\def\cal#1{\mathcal{#1}}%

\def\Ind#1#2{#1\setbox0=\hbox{$#1x$}\kern\wd0\hbox to 0pt{\hss$#1\mid$\hss} \lower.9\ht0\hbox to 0pt{\hss$#1\smile$\hss}\kern\wd0} 
\def\Notind#1#2{#1\setbox0=\hbox{$#1x$}\kern\wd0\hbox to 0pt{\mathchardef \nn="3236\hss$#1\nn$\kern1.4\wd0\hss}\hbox to 0pt{\hss$#1\mid$\hss}\lower.9\ht0 \hbox to 0pt{\hss$#1\smile$\hss}\kern\wd0} 
\def\nind{\mathop{\mathpalette\Notind{}}} 

\global\long\def\ind{\mathop{\mathpalette\Ind{}}}%
\global\long\def\dom{\operatorname{Dom}}%
 
\global\long\def\nind{\mathop{\mathpalette\Notind{}}}%
\global\long\def\average#1#2#3{Av_{#3}\left(#1/#2\right)}%
\global\long\def\Ff{\mathfrak{F}}%
\global\long\def\mx#1{Mx_{#1}}%
\global\long\def\maps{\mathfrak{L}}%

\global\long\def\Esat{E_{\mbox{sat}}}%
\global\long\def\Ebnf{E_{\mbox{rep}}}%
\global\long\def\Ecom{E_{\mbox{com}}}%
\global\long\def\BtypesA{S_{\Bb}^{x}\left(A\right)}%

\global\long\def\init{\trianglelefteq}%
\global\long\def\fini{\trianglerighteq}%
\global\long\def\Bb{\cal B}%
\global\long\def\Lim{\operatorname{Lim}}%
\global\long\def\supp{\operatorname{supp}}%

\global\long\def\SquareClass{\cal M}%
\global\long\def\leqstar{\leq_{*}}%
\global\long\def\average#1#2#3{Av_{#3}\left(#1/#2\right)}%
\global\long\def\cut#1{\mathfrak{#1}}%
\global\long\def\measureSpace#1#2{\mathfrak{M}_{#1}\left(#2\right)}%
\global\long\def\NTPT{\text{NTP}_{2}}%
\global\long\def\Zz{\mathbb{Z}}%
\global\long\def\TPT{\text{TP}_{2}}%
\global\long\def\Av{\operatorname{Av}}%

\global\long\def\OurSequence{\mathcal{I}}%
\global\long\def\Ff{\mathbb{F}}%
\global\long\def\Cc{\mathbb{C}}%

\title{Some NIP-like phenomena in NTP$_{2}$ }
\author{Itay Kaplan and Pierre Simon}
\begin{abstract}
We introduce the notion of an NTP$_{2}$-smooth measure and prove
that they exist assuming NTP$_{2}$. Using this, we propose a notion
of distality in NTP$_{2}$ that unfortunately does not intersect simple
theories trivially. We then prove a finite alternation theorem for
a subclass of NTP$_{2}$ that contains resilient theories. In the
last section we prove that under NIP, any type over a model of singular
size is finitely satisfiable in a smaller model, and ask if a parallel
result (with non-forking replacing finite satisfiability) holds in
NTP$_{2}$.
\end{abstract}

\thanks{The first author would like to thank the Israel Science Foundation
for partial support of this research (Grants no. 1533/14, 804/22). }
\thanks{Partially supported by ValCoMo (ANR-13-BS01-0006), NSF (grant DMS
1665491), and the Sloan foundation.}
\address{Itay Kaplan \\
The Hebrew University of Jerusalem\\
Einstein Institute of Mathematics \\
Edmond J. Safra Campus, Givat Ram\\
Jerusalem 91904, Israel}
\email{kaplan@math.huji.ac.il}
\urladdr{https://itaykaplan.huji.ac.il/}
\address{Pierre Simon\\
UC Berkeley\\
Mathematics Department, Evans Hall\\
Berkeley, CA, 94720-3840\\
USA}
\email{pierre.simon@berkeley.edu}
\urladdr{http://www.normalesup.org/\textasciitilde simon/}
\maketitle

\section{Introduction}

In recent years, a lot of attention has been given to unstable classes
of first order theories. In particular, NIP and to a lesser extent
NTP$_{2}$. The former, NIP, is a very important class of theories
which was studied extensively, see \cite{pierrebook}. The latter,
NTP$_{2}$, is a class of theories which contains both simple and
NIP theories. Examples include the ultraproduct of the $p$-adics
\cite{MR3129735}, bounded PRC fields \cite{MR3564381} and valued
fields with a generic automorphism \cite{MR3273459}. Though it is
a very large class of theories, some general nontrivial results were
nonetheless attained. For example, in \cite{Kachernikov} it is proved
that forking and dividing agree over models, and \cite{MR3226015}
contains an independence theorem for NTP$_{2}$. Under the assumption
of groups or fields, more has been done. See for example \cite{MR3272764}
(about groups and fields in general NTP$_{2}$), \cite{MR3625116}
(about definable envelopes of subgroups), and more recently \cite{Montenegro2016}
(about groups definable in bounded PRC fields). 

Roughly speaking the ideology guiding our results on NTP$_{2}$ is
that it is NIP for forking. We exhibit this in two instances. 

1) In Section \ref{sec:On-NTP-smooth-measures} we introduce the notion
of an NTP$_{2}$-smooth Keisler measure. For any theory $T$ and $M\models T$,
a Keisler measure $\mu$ over $M$ is \emph{smooth} if for every $N\succ M$,
there is a \emph{unique }extension of $\mu$ to $\measureSpace xN$
(the space of all measures over $N$, see Definition \ref{def:finitely additive measure}
below). This notion turned out to be very important in the study of
measures in NIP theories (see \cite[Section 7.3]{pierrebook}) so
it is natural to find a parallel notion for NTP$_{2}$. As per our
guiding ideology, we say that $\mu\in\measureSpace xM$ is NTP$_{2}$-smooth
if for every extension $\mu'\in\measureSpace xN$ of $\mu$, if $\varphi\left(x,a\right)$
divides (equivalently forks, see Remark \ref{rem:dividing iff forking for smooth measures})
over $M$ then $\mu'\left(\varphi\left(x,a\right)\right)=0$. We then
prove that such measures exist: every Keisler measure over a model
in an NTP$_{2}$ theory can be extended to an NTP$_{2}$-smooth one.
In the last part of this section, Subsection \ref{subsec:NTP-distal},
we try to define a suitable notion of NTP$_{2}$-distality, and provide
two equivalent definitions (one of them using NTP$_{2}$-smooth measures).
However, it intersects simple theories, and thus this definition will
probably have to be refined. 

2) In Section \ref{sec:finite alternation} we prove a finite alternation
result. Recall that in NIP, just by definition, there is no indiscernible
sequence $\sequence{a_{i}}{i<\omega}$, a formula $\varphi\left(x,y\right)$
and $b$ such that $\varphi\left(b,a_{i}\right)$ holds iff $i$ is
even. We introduce a class of theories, which we call $\omega$-resilient
theories (containing resilient theories), and prove that assuming
both NTP$_{2}$ and $\omega$-resilience, if $\sequence{a_{i}}{i<\omega}$
is an indiscernible sequence and $\varphi\left(x,b\right)$ divides
over $I=\sequence{a_{2i}}{i<\omega}$, then for all but finitely many
$i$'s, $\neg\varphi\left(a_{i},b\right)$ holds. Note that this holds
if $T$ is simple (see just below Theorem \ref{thm:finite alternation in NTP2}). 

Finally, in Section \ref{sec:On-singular-local} we move to NIP theories,
and prove that a kind of local character result holds there, which
we call ``singular local character''. Namely, if $p\in S\left(M\right)$
and $\left|M\right|$ is singular with cofinality greater than $\left|T\right|$,
then $p$ is finitely satisfiable over $N\prec M$ of smaller cardinality.
In particular, $p$ does not fork over $N$. Since the last statement
is trivially true for simple theories, it is natural to ask whether
this is true for NTP$_{2}$ (see Question \ref{que:generalizing theorem to NTP2}). 

We would like to thank Saharon Shelah for Example \ref{exa:Random graph not local character for coheir}
and for his help with Theorem \ref{thm:singular local character over one model}
and in particular for turning \textquotedblleft non-dividing\textquotedblright{}
into \textquotedblleft finitely satisfiable\textquotedblright . We
would also like to thank the anonymous referee for their extremely
careful reading and useful comments, and Timo Krisam for his comments. 

\section{Preliminaries}

We recall the basic definitions of NIP and NTP$_{2}$. For a thorough
discussion of NIP and its importance, we refer the reader to \cite{pierrebook}.
The class NTP$_{2}$ is also discussed there, but we also add \cite{MR3129735}. 
\begin{defn}
A complete theory $T$ is \emph{NIP} if there is no formula $\varphi\left(x,y\right)$
with the \emph{independence property} (\emph{IP}), where $\varphi$
has IP if, in some $M\models T$ there are $\sequence{a_{i}}{i<\omega}$
and $\sequence{b_{s}}{s\subseteq\omega}$ such that $M\models\varphi\left(a_{i},b_{s}\right)$
iff $i\in s$. 
\end{defn}

\begin{defn}
A formula $\varphi\left(x,y\right)$ has the \emph{tree property of
the second kind} (\emph{TP}$_{2}$) if there is an array $\sequence{a_{i,j}}{i,j<\omega}$
and some $k<\omega$ such that every vertical path is consistent (for
every $\eta:\omega\to\omega$, $\set{\varphi\left(x,a_{i,\eta\left(i\right)}\right)}{i<\omega}$
is consistent) and every row is $k$-inconsistent ($\set{\varphi\left(x,a_{i,j}\right)}{j<\omega}$
is $k$-inconsistent). A complete theory $T$ is \emph{NTP}$_{2}$
if no formula has TP$_{2}$. 
\end{defn}

Our notations are standard, e.g., $T$ will denote some complete first-order
theory and $\C$ will be its monster model. 

\section{\label{sec:On-NTP-smooth-measures}On NTP$_{2}$-smooth measures
and a possible definition for NTP$_{2}$-distal theories}

Recall that a \emph{finitely additive probability measure} on a Boolean
algebra $B$ is a function $\mu:B\to\left[0,1\right]$ such that $\mu\left(1\right)=1$,
$\mu\left(x^{c}\right)=1-\mu\left(x\right)$ and $\mu\left(x\vee y\right)=\mu\left(x\right)+\mu\left(y\right)$
whenever $x\wedge y=0$. 
\begin{defn}
\label{def:finitely additive measure}Suppose that $A$ is a set of
parameters in some model $M$. A \emph{Keisler measure }(or just a
\emph{measure}) over $A$ in the variable $x$ is a finitely additive
probability measure on $L_{x}\left(A\right)$: the Boolean algebra
of definable sets in $x$ over $A$. We denote the space of measures
in $x$ over $A$ by $\measureSpace xA$. 
\end{defn}

\begin{fact}
\label{fact:extending of measures} \cite[Lemma 7.3]{pierrebook}Let
$\Omega\subseteq L_{x}\left(A\right)$ be a sub-algebra of $L_{x}\left(A\right)$
(i.e., closed under intersection, union and complement, and contains
$x=x$). Let $\mu$ be a finitely additive probability measure on
$\Omega$. Then $\mu$ extends to a Keisler measure over $A$. 
\end{fact}

\subsection{NTP$_{2}$-smooth measures}

In this section we will define an analogous notion to smooth measures
from NIP in the NTP$_{2}$-context. The main result is that every
measure can be extended to an NTP$_{2}$-smooth measure, assuming
NTP$_{2}$. 
\begin{rem}
\label{rem:smooth measures in NIP}Recall (see \cite[Definition 7.7]{pierrebook})
that if $M\models T$ then $\mu\in\measureSpace xM$ is \emph{smooth}
if for every $N\succ M$, there is a \emph{unique }extension of $\mu$
to $\measureSpace xN$. If $\mu\in\measureSpace xN$ and $M\prec N$
then $\mu$ is \emph{smooth over $M$} if the restriction $\mu|_{M}$
is smooth. We can also extend this definition to any set of parameters,
working in $\C$: $\mu\in\measureSpace xA$ is smooth if it has a
unique extension to $\measureSpace x{\C}$. 
\end{rem}

\begin{rem}
\label{rem:invariant measures}For a global measure $\mu\in\measureSpace x{\C}$,
$\mu$ is called \emph{$A$-invariant} for some set $A$ if whenever
$b\equiv_{A}c$ we have $\mu\left(\varphi\left(x,b\right)\right)=\mu\left(\varphi\left(x,c\right)\right)$.
In NIP, by \cite[Proposition 7.15]{pierrebook}, $\mu\in\measureSpace x{\C}$
is invariant over a model $M$ iff for every formula $\varphi\left(x,c\right)$
which forks (or divides) over $M$, $\mu\left(\varphi\left(x,c\right)\right)=0$
(in this case we say that $\mu$ does not fork over $M$). 
\end{rem}

\begin{defn}
\label{def:NTP2 smooth}A Keisler measure $\mu\in\measureSpace xA$
is called \emph{NTP$_{2}$-smooth} if for every $A\subseteq N$ and
any extension of $\mu$ to $\mu'\in\measureSpace xN$, $\mu'$ does
not divide over $A$: if $\varphi\left(x,b\right)$ divides over $A$
then $\mu'\left(\varphi\left(x,b\right)\right)=0$. When $\mu\in\measureSpace xN$
and $A\subseteq N$ we say that $\mu$ is\emph{ NTP$_{2}$-smooth}
over $A$ if $\mu|_{A}$ is NTP$_{2}$-smooth. 
\end{defn}

\begin{rem}
\label{rem:dividing iff forking for smooth measures}A measure $\mu\in\measureSpace xA$
is NTP$_{2}$-smooth iff for every $A\subseteq N$ and any extension
of $\mu$ to $\mu'\in\measureSpace xN$, if $\varphi\left(x,b\right)\in L_{x}\left(N\right)$
\emph{forks} over $A$ then $\mu'\left(\varphi\left(x,b\right)\right)=0$.
To see this, note that if $\varphi\left(x,b\right)$ forks over $A$,
then we can extend the measure $\mu'$ to include in its domain the
dividing formulas whose disjunction $\varphi\left(x,b\right)$ implies,
so all of these must have measure zero. 
\end{rem}

\begin{fact}
\label{fact:poitive indiscernible is consistent}\cite[Lemma 7.5]{pierrebook}If
$M$ is a model, $\sequence{b_{i}}{i<\omega}$ is an indiscernible
sequence of tuples in $M$ and $\mu\in\measureSpace xM$ is such that
$\mu\left(\varphi\left(x,b_{i}\right)\right)\geq\varepsilon>0$ for
all $i$, then $\set{\varphi\left(x,b_{i}\right)}{i<\omega}$ is consistent. 
\end{fact}

\begin{prop}
\label{prop:in NIP smooth =00003D NTP2 smooth}If $T$ is NIP and
$M\models T$, then $\mu\in\measureSpace xM$ is smooth iff it is
NTP$_{2}$-smooth.
\end{prop}

\begin{proof}
Suppose that $\mu$ is smooth, and $\mu'$ extends $\mu$ to $\measureSpace xN$,
and $\mu'\left(\varphi\left(x,b\right)\right)>0$ with $\varphi\left(x,b\right)$
dividing over $M$. Assuming that $N$ is $\left|M\right|^{+}$-saturated,
it contains an indiscernible sequence $\sequence{b_{i}}{i<\omega}$
over $M$ which witnesses dividing. As $\mu$ is smooth, it follows
that $\mu'\left(\varphi\left(x,b_{i}\right)\right)=\mu'\left(\varphi\left(x,b\right)\right)$
for all $i$ (otherwise, by applying an automorphism fixing $M$ taking
$b_{i}$to $b$ we get measures extending $\mu$ giving distinct values
to $\varphi\left(x,b\right)$). Together we get a contradiction to
Fact \ref{fact:poitive indiscernible is consistent}. This direction
does not need NIP nor that $M$ is a model.

On the other hand, if $\mu$ is NTP$_{2}$-smooth then it is smooth:
first, by \cite[Proposition 7.9]{pierrebook}, we know that $\mu$
can be extended to a smooth measure in $\measureSpace xN$ for some
$N\succ M$. Extend it even further to $\mu'\in\measureSpace x{\C}$.
By definition $\mu'$ is NTP$_{2}$-smooth over $M$ and by NIP (and
Remark \ref{rem:invariant measures}) $\mu'$ is $M$-invariant. Now,
\cite[Lemma 7.17]{pierrebook} tells us that that $\mu'$ is smooth
over $M$ so we are done.
\end{proof}
As was said in the proof of Proposition \ref{prop:in NIP smooth =00003D NTP2 smooth},
in NIP, every measure can be extended to a smooth one. The analogous
statement in NTP$_{2}$ is then the following. 
\begin{thm}
\label{thm:every measure can be extended to a smooth }Suppose that
$T$ is NTP$_{2}$ and $M\models T$. Any Keisler measure $\mu\in\measureSpace xM$
can be extended to an NTP$_{2}$-smooth measure over some $N\succ M$. 
\end{thm}

\begin{proof}
Suppose not. Construct an increasing continuous sequence of measures
and models 
\[
\sequence{\left(M_{\alpha},\mu_{\alpha}\right)}{\alpha<\left(\left|T\right|+2^{\aleph_{0}}\right)^{+}}
\]
 such that $M_{0}=M,\mu_{0}=\mu$ and for every $\alpha<\left(\left|T\right|+2^{\aleph_{0}}\right)^{+}$,
there are $\varphi_{\alpha}\left(x,y_{\alpha}\right)\in L$ and $b_{\alpha}\in M_{\alpha+1}$
such that $\varphi_{\alpha}\left(x,b_{\alpha}\right)$ divides over
$M_{\alpha}$ and $\mu_{\alpha+1}\left(\varphi_{\alpha}\left(x,b_{\alpha}\right)\right)=\varepsilon_{\alpha}>0$.
Also, for each $\alpha$ there is some $M_{\alpha}$-indiscernible
sequence $\bar{b}_{\alpha}=\sequence{b_{\alpha,i}}{i\in\Zz}\subseteq M_{\alpha+1}$
such that $b_{\alpha,0}=b_{\alpha}$ and $\set{\varphi\left(x,b_{\alpha,i}\right)}{i\in\Zz}$
is $k_{\alpha}$-inconsistent. 
\begin{claim}
\label{claim:wlog}In the construction we can make sure that the following
holds.
\end{claim}

\begin{itemize}
\item [$\left(\star\right)$]For each formula $\psi\left(x,y_{\alpha}\right)$
over $M_{\alpha}$, and every $i,j\in\Zz$, if $\mu_{\alpha+1}\left(\psi\left(x,b_{\alpha,i}\right)\right)>\varepsilon$
then $\mu_{\alpha+1}\left(\psi\left(x,b_{\alpha,j}\right)\right)\geq\varepsilon\cdot2^{-\left|i-j\right|-1}$.
In particular, $\mu_{\alpha+1}\left(\psi\left(x,b_{\alpha,i}\right)\right)>0$
iff $\mu_{\alpha+1}\left(\psi\left(x,b_{\alpha,j}\right)\right)>0$. 
\end{itemize}
\begin{proof}[Proof of Claim]
 By our assumption towards contradiction, in stage $\alpha<\left(\left|T\right|+2^{\aleph_{0}}\right)^{+}$
of the construction we can find $M_{\alpha+1}'$, $\varphi_{\alpha}$,
$b_{\alpha}$, $\sequence{b_{\alpha,i}}{i\in\Zz}$, $\varepsilon_{\alpha}$
and some $\mu_{\alpha+1}'$ which satisfy everything except $\left(\star\right)$.
To get $\left(\star\right)$, we let $\sigma\in\Aut\left(\C/M_{\alpha}\right)$
take $\sequence{b_{\alpha,i}}{i\in\Zz}$ to $\sequence{b_{\alpha,i+1}}{i\in\Zz}$,
and extend $M_{\alpha+1}'$ to $M_{\alpha+1}\succ M_{\alpha+1}'$
such that $\sigma\restriction M_{\alpha+1}\in\Aut\left(M_{\alpha+1}/M_{\alpha}\right)$.
Additionally, extend $\mu_{\alpha+1}'$ to $\measureSpace x{M_{\alpha+1}}$.
Now let 
\[
\mu_{\alpha+1}=\sum\set{2^{-\left|i\right|-2}\sigma^{i}\left(\mu_{\alpha+1}'\right)}{i\in\Zz\backslash\left\{ 0\right\} }+1/2\mu_{\alpha+1}'.
\]
Note that $\mu_{\alpha+1}\in\measureSpace x{M_{\alpha+1}}$ and it
extends $\mu_{\alpha}$. Let us check that $\left(\star\right)$ holds.
Suppose that $\mu_{\alpha+1}\left(\psi\left(x,b_{\alpha,i}\right)\right)>\varepsilon$.
Then without loss of generality 
\[
\sum_{k=0}^{\infty}2^{-k-2}\mu'_{\alpha+1}\left(\psi\left(x,b_{\alpha,k+i}\right)\right)>2^{-1}\varepsilon
\]
 (this is the ``positive side'' of this sum\footnote{In general the way the action of an automorphism $\sigma$ on a measure
$\mu$ is defined is via the expression $\sigma\left(\mu\right)\left(\varphi\right)=\mu\left(\sigma^{-1}\left(\varphi\right)\right)$,
and here ``positive'' refers to the indices in the indiscernible
sequence. }). If $j<i$ then the positive side of the sum which calculates $\mu_{\alpha+1}\left(\psi\left(x,b_{\alpha,j}\right)\right)$
is 
\[
\geq\sum_{k=i-j}^{\infty}2^{-k-2}\mu_{\alpha+1}'\left(\psi\left(x,b_{\alpha,k+j}\right)\right)>2^{j-i-1}\varepsilon.
\]

If $j>i$, then as 
\[
\mu_{\alpha+1}\left(\psi\left(x,b_{\alpha,j}\right)\right)\geq\sum_{k=i-j}^{\infty}2^{-\left|k\right|-2}\mu'_{\alpha+1}\left(\psi\left(x,b_{\alpha,k+j}\right)\right),
\]
 we get that it is 
\[
\geq2^{i-j}\sum_{k=0}^{j-i-1}2^{k-2}\mu'_{\alpha+1}\left(\psi\left(x,b_{\alpha,k+i}\right)\right)+2^{j-i}\cdot\sum_{k=j-i}^{\infty}2^{-k-2}\mu'_{\alpha+1}\left(\psi\left(x,b_{\alpha,k+i}\right)\right)
\]
 which is 
\[
\geq2^{i-j}\sum_{k=0}^{\infty}2^{-k-2}\mu'_{\alpha+1}\left(\psi\left(x,b_{\alpha,k+i}\right)\right)>2^{i-j-1}\varepsilon.
\]

This concludes the proof of the claim. 
\end{proof}
Now extract a sequence $\sequence{\left(M_{i},\mu_{i},\bar{b}_{i}\right)}{i<\omega}$
such that for some fixed formula $\varphi\left(x,y\right)$, $\varphi\left(x,b_{i,0}\right)$
divides over $M_{i}$ and even $k$-divides for a fixed $k$ as witnessed
by $\bar{b}_{i}$, and $\mu_{i+1}\left(\varphi\left(x,b_{i,0}\right)\right)>\varepsilon_{*}$
for some fixed $\varepsilon_{*}$. Let $\mu^{*}=\bigcup_{i<\omega}\mu_{i}$.

Next we extract a sequence $\sequence{\bar{b}'_{i}}{i<\omega}$ where
$\bar{b}'_{i}=\sequence{b'_{i,j}}{j\in\Zz}$ such that $\sequence{\bar{b}_{i}'}{i<\omega}$
is indiscernible with respect to $\mu^{*}$ realizing the $\mu^{*}$-EM-type
of $\sequence{\bar{b}_{i}}{i<\omega}$. Indiscernible with respect
to $\mu^{*}$ means this sequence is indiscernible and for all $i_{0}<\ldots<i_{n-1}$,
$\mu^{*}\left(\psi\left(x,\bar{b}'_{i_{1}},\ldots,\bar{b}_{i_{n}}'\right)\right)=\mu^{*}\left(\psi\left(x,\bar{b}_{0}',\ldots,\bar{b}_{n-1}'\right)\right)$.
Realizing the $\mu^{*}$-EM-type means realizing the EM-type and moreover,
if $\mu^{*}\left(\psi\left(x,\bar{b}_{0}',\ldots,\bar{b}_{n-1}'\right)\right)=\varepsilon$,
then for every $\delta>0$ there is an increasing tuple $i_{0}<\ldots<i_{n-1}<\omega$
with $\left|\mu^{*}\left(\psi\left(x,\bar{b}_{i_{0}},\ldots,\bar{b}_{i_{n-1}}\right)\right)-\varepsilon\right|<\delta$. 

How to construct such a sequence is explained in \cite[after Lemma 7.4 and proof of Lemma 7.5]{pierrebook}:
this can be seen as the usual application of Ramsey and compactness
to an expanded structure, where one adds a sort for the interval $\left[0,1\right]$
with addition and the order and function symbols $f_{\varphi}$ for
every formula $\varphi\left(x,y\right)$ taking a tuple $c$ to the
measure of $\varphi\left(x,c\right)$. 

We forgot $M_{i}$, but we still retain that $\set{\varphi\left(x,b_{i,j}'\right)}{j\in\Zz}$
is inconsistent (because we fixed $k$) and $\left(\star\right)$
of Claim \ref{claim:wlog} still holds (for $\varphi\left(x,y\right)$
defined over $b'_{<i}$). Rename $b_{i,j}'$ to $b_{i,j}$. 

Suppose that there was some $K<\omega$ with $\mu^{*}\left(\bigwedge_{i<K}\varphi\left(x,b_{i,i}\right)\right)>0$
but $\mu^{*}\left(\bigwedge_{i<K+1}\varphi\left(x,b_{i,i}\right)\right)=0$.
Then by $\left(\star\right)$ of Claim \ref{claim:wlog}, $\mu^{*}\left(\bigwedge_{i<K}\varphi\left(x,b_{i,i}\right)\wedge\varphi\left(x,b_{K,0}\right)\right)=0$,
so letting $\bar{c}_{l}=\sequence{b_{l+i,i}}{i<K}$ (for $l<\omega$)
and $\psi\left(x,\bar{y}\right)=\bigwedge_{i<K}\varphi\left(x,y_{i}\right)$
we get that $\sequence{\mu^{*}\left(\psi\left(x,\bar{c}_{Kl}\right)\right)}{l<\omega}$
is constant and positive, while $\mu^{*}\left(\psi\left(x,\bar{c}_{lK}\right)\wedge\psi\left(x,\bar{c}_{\left(l+1\right)K}\right)\right)=0$
for all $l<\omega$, but the total measure is 1 so this is impossible.

By NTP$_{2}$ and compactness there is some $N<\omega$ such that
there is no array $\sequence{a_{i,j}}{i<N,j<N}$ such that for every
$i<N$, $\set{\varphi\left(x,a_{i,j}\right)}{j<N}$ is $k$-inconsistent
and for every $\eta:N\to N$, $\set{\varphi\left(x,a_{i,\eta\left(i\right)}\right)}{i<N}$
is consistent. 

Suppose that the measure of the diagonal $\mu^{*}\left(\bigwedge_{i<N}\varphi\left(x,b_{i,i}\right)\right)$
is positive. 

In this case, by $\mu^{*}$-indiscernibility of $\sequence{\bar{b}_{i}}{i<\omega}$
and Fact \ref{fact:poitive indiscernible is consistent} (which is
applied to the sequence $\sequence{\sequence{b_{lN+i,i}}{i<N}}{l<\omega}$),
it follows that the set of all $N$-diagonals $\set{\bigwedge_{i<N}\varphi\left(x,b_{lN+i,i}\right)}{l<\omega}$
is consistent. In particular, for any $\eta:N\to N$, $\set{\varphi\left(x,b_{lN+\eta\left(l\right),\eta\left(l\right)}\right)}{l<N}$
is consistent (because this is a subset: we choose one element from
each diagonal). Hence, by indiscernibility of $\sequence{\bar{b}_{i}}{i<\omega}$
it follows that $\set{\varphi\left(x,b_{l,\eta\left(l\right)}\right)}{l<N}$
is consistent. 

This is a contradiction to the choice of $N$, hence $\mu^{*}\left(\bigwedge_{i<N}\varphi\left(x,b_{i,i}\right)\right)=0$.
But $\mu^{*}\left(\varphi\left(x,b_{0,0}\right)\right)>\varepsilon_{*}$,
so we can find some $K$ as above --- contradiction.
\end{proof}

\subsection{\label{subsec:NTP-distal}On a possible definition of NTP$_{2}$-distal
theories}

Distal theories form an important class of NIP theories. Defined and
studied in \cite{Distal}, they were studied further in \cite{Chernikov2015}
where some surprising combinatorial results were discovered. Distal
theories were given a ``set-theoretic'' characterization in terms
of the existence of saturated models in \cite{MR3651210}. 

We would like to suggest a possible definition of NTP$_{2}$-distal.
In the context of NIP, several equivalent definitions of distality
can be given. We will use the one which relates it to smooth measures.
We will see in the end that our proposed definition lacks an important
property of distal theories: in the NIP context, distal theories can
never be stable. Here we would like to have that NTP$_{2}$-distal
theories are never simple. This is not the case in our definition,
which raises the question of possible refinements. We will not deal
with this here. 

First let us give the more familiar definition. 
\begin{defn}
\label{def:distal}A theory $T$ is \emph{distal} if whenever $I_{1}+a+I_{2}$
is indiscernible, $I_{1}$, $I_{2}$ have no endpoints and $I_{1}+I_{2}$
is indiscernible over $A$, $I_{1}+a+I_{2}$ is indiscernible over
$A$. 
\end{defn}

Note that if $T$ is distal then $T$ is NIP (if not, then we can
find a formula $\varphi\left(x,y\right)$, an indiscernible sequence
$\sequence{a_{i}}{i<\omega}$ and $b$ such that $\varphi\left(a_{i},b\right)$
holds iff $i$ is even. Extracting, we may assume that the sequence
of pairs $\sequence{\left(a_{2i},a_{2i+1}\right)}{i<\omega}$ is indiscernible
over $b$. By compactness, we can find such a sequence of any order
type in which every element has a successor, and thus it is easy to
get a contradiction to distality).  

Given an indiscernible sequence $I=\sequence{a_{i}}{i\in\left[0,1\right]}$
where $a_{i}$ has the same length as the variable $x$, let $\Omega\subseteq L_{x}\left(\C\right)$
be the family of all definable sets $D\subseteq\C$ such that $D_{I}=\set{i\in\left[0,1\right]}{a_{i}\in D}$
is a Borel measurable set. Note that $\Omega$ is a Boolean algebra
(i.e., closed under intersections, unions and complements, and contains
$x=x$). Let $\lambda$ be the Lebesgue measure on $\left[0,1\right]$.
Then $\lambda$ induces a probability measure $\Av_{I}$ on $\Omega$
by setting $\Av_{I}\left(D\right)=\lambda\left(D_{I}\right)$. Note
that if $D$ is definable over $I$, then $D_{I}$ is a finite union
of intervals, so Borel, hence $\Av_{I}$ defines a Keisler measure
on $L_{x}\left(I\right)$, which we will naturally denote by $\Av_{I}|_{I}$
(the restriction of $\Av_{I}$ to $I$). 

In NIP theories, every $D_{I}$ is a finite union of intervals, so
that $\Av_{I}$ is a global Keisler measure. 
\begin{fact}
\label{fact:Distal iff all averages are smooth} \cite[Proposition 2.21]{Distal}A
complete theory $T$ is distal iff it is NIP and for all indiscernible
sequences $I=\sequence{a_{i}}{i\in\left[0,1\right]}$ and every model
$M$, $\Av_{I}|_{M}$ is smooth. 
\end{fact}

In fact, reading the proof in \cite[Proposition 2.21]{Distal}, we
get that $T$ is distal iff it is NIP and for all such $I$'s, $\Av_{I}|_{I}$
is smooth. Furthermore, NIP follows (so we don't need to assume it).
To see this directly, suppose that $\varphi\left(x,y\right)$ has
IP and that $I'=\sequence{a_{i}}{i\in\left[0,1\right]\times\left\{ 0,1\right\} }$
is an indiscernible sequence in $x$, ordered lexicographically, such
that for some $b$, $\varphi\left(a_{i},b\right)$ holds if and only
if the second coordinate of $i$ is 0. Now let $I=\sequence{a_{i}}{i\in\left[0,1\right]\times\left\{ 0\right\} }$.
Now we will use the following fact on extensions of measures.
\begin{fact}
\label{fact:extending measures}\cite[Lemma 7.4]{pierrebook} If $A$
is any set, $\mu\in\measureSpace xA$ and $\varphi\left(x,b\right)$
is some formula over $\C$, then for every $r\in\left[0,1\right]$
such that 
\[
\inf\set{\mu\left(\theta\right)}{\theta\in L_{x}\left(A\right),\theta\vdash\varphi\left(x,b\right)}\leq r\leq\sup\set{\mu\left(\theta\right)}{\theta\in L_{x}\left(A\right),\varphi\left(x,b\right)\vdash\theta}
\]
 there is an extension $\nu$ of $\mu$ to $\measureSpace x{\C}$
such that $\nu\left(\varphi\left(x,b\right)\right)=r$. 
\end{fact}

(Actually, the lemma in \cite[Lemma 7.4]{pierrebook} is stated only
when $A$ is a model, but the same proof goes through.)

By Fact \ref{fact:extending measures}, there are $2^{\aleph_{0}}$
extensions of $\Av_{I}|_{I}$: the infimum on the left is 0 (because
if $\theta\in L_{x}\left(I\right)$, $\theta\vdash\varphi\left(x,b\right)$
and $\Av\left(\theta\right)>0$ then $\theta$ contains an interval
in $I$, but then also an interval in $I'$, contradicting the choice
of $b$) while the supremum on the right is 1 (by considering the
negation of the formula $\theta$ in the argument for the infimum).

Thus we propose the following definition.
\begin{defn}
\label{def:NTP2 distal}Say that a theory $T$ is \emph{NTP$_{2}$-distal}
if it is NTP$_{2}$ and for every indiscernible sequence $I=\sequence{a_{i}}{i\in\left[0,1\right]}$,
$\Av_{I}|_{I}$ is NTP$_{2}$-smooth. 
\end{defn}

\begin{question}
Do we need to assume NTP$_{2}$ in Definition \ref{def:NTP2 distal}?
\end{question}

As with distal theories, we would like to have an equivalent definition
that does not use measures. 
\begin{thm}
\label{thm:NTP2 distal in terms of sequence} An NTP$_{2}$ theory
$T$ is NTP$_{2}$-distal iff for every dense indiscernible sequence
$I$ which we write as $I_{1}+I_{2}$, with $I_{1}$ endless and $I_{2}$
with no first element, if $I$ is $d$-indiscernible, $I_{1}+b+I_{2}$
is indiscernible and $\varphi\left(x,d\right)$ divides over $I$
then $\neg\varphi\left(b,d\right)$ holds.

Moreover, if $\varphi\left(x_{0},\ldots,x_{n-1},y\right)$ is a formula
over $I$ such that $\varphi\left(x_{0},\ldots,x_{n-1},d\right)$
divides over $I$ then $\neg\varphi$ is in the EM-type of $I_{1}+b+I_{2}$
over $dI$ (meaning that for all increasing tuples $a_{0},\ldots,a_{n-1}$
from $I_{1}+b+I_{2}$, $\varphi\left(a_{0},\ldots,a_{n-1}\right)$
does not hold). 
\end{thm}

\begin{proof}
Suppose first that $T$ is NTP$_{2}$-distal and let $I=I_{1}+I_{2}$,
$b$ and $\varphi\left(x,d\right)$ as in there and assume that $\varphi\left(b,d\right)$
holds. Let $k<\omega$ witness that $\varphi\left(x,d\right)$ divides
over $I$. We may assume that $I$ is countable (by taking a subsequence)
and ordered by $\left(0,1\right)\cap\Qq$ where $b$ is in some irrational
spot $\alpha\in\left(0,1\right)$. By compactness, we can extend $I$
to have order type $\left[0,1\right]\backslash\left\{ \alpha\right\} $
where $I_{1}=I_{<\alpha}$ and $I_{2}=I_{>\alpha}$ (note that the
conditions involved, that $I$ is $d$-indiscernible, that $\varphi\left(x,d\right)$
$k$-divides over $I$ and that $I_{1}+b+I_{2}$ is indiscernible
are type definable). Write $I=\sequence{a_{i}}{i\in\left[0,1\right]\backslash\left\{ \alpha\right\} }$.
Let $a_{\alpha}$ be such that $I'=I_{1}+a_{\alpha}+I_{2}$ is $d$-indiscernible
and $\varphi\left(x,d\right)$ divides over $I'$. Then by assumption,
$\Av_{I'}|_{I'}$ is NTP$_{2}$-smooth, which means that in every
extension $\mu$ of $\Av_{I'}|_{I'}$, $\mu\left(\varphi\left(x,d\right)\right)=0$.
By Fact \ref{fact:extending measures}, for every $0<\varepsilon$
there is some formula $\theta_{\varepsilon}\left(x\right)\in L_{x}\left(I'\right)$
such that $\varphi\left(x,d\right)\vdash\theta_{\varepsilon}$ and
$\Av_{I'}\left(\theta_{\varepsilon}\right)<\varepsilon$. Take $\varepsilon$
small enough so that $0<\alpha-\varepsilon<\alpha+\varepsilon<1$.
The formula $\theta_{\varepsilon}$ is over $I'$ so it defines a
union of intervals in $I'$, each with endpoints in the parameters
defining $\theta_{\varepsilon}$. Now we want to move those parameters
so that $\theta_{\varepsilon}$ is over $I_{<\alpha-\varepsilon}+I_{>\alpha+\varepsilon}$.
Since $I'$ is indiscernible over $d$, we will still have that $\varphi\left(x,d\right)\vdash\theta_{\varepsilon}$.
The challenge is to do it without increasing the measure. How? If
there are no indices in the ``forbidden zone'' $\left[\alpha-\varepsilon,\alpha+\varepsilon\right]$,
we are done. Otherwise we proceed by induction on the number of indices
there. Suppose $i$ is the leftmost index in the forbidden zone. Recall
that $\theta_{\varepsilon}$ defines a finite union of intervals in
$\left[0,1\right]$. If $i$ is an isolated point of this union, then
moving $a_{i}$ to $I_{<\alpha-\varepsilon}$ doesn't change the measure
of $\theta_{\varepsilon}$. If there is no open interval of the form
$\left(i,i+\delta\right)$ contained in this union, again move $a_{i}$
to $I_{<\alpha-\varepsilon}$ which may only decrease the measure.
If there is no open interval of the form $\left(i-\delta,i\right)$
there, move $a_{i}$ and all the points in the forbidden zone to $I_{>\alpha+\varepsilon}$
which again will only decrease the measure. Lastly, if there are intervals
in both sides of $i$, moving $a_{i}$ to $I_{<\alpha-\varepsilon}$
will not change the measure. Since $\varphi\left(b,d\right)$ holds,
$\theta_{\varepsilon}\left(b\right)$ holds, and as $I_{1}+b+I_{2}$
is indiscernible, $\theta_{\varepsilon}\left(a_{i}\right)$ holds
for all $i\in\left[\alpha-\varepsilon,\alpha+\varepsilon\right]$.
Hence $\Av_{I'}\left(\theta_{\varepsilon}\right)\geq2\varepsilon$
--- contradiction. 

On the other hand, suppose that that the condition on the right hand
side holds and we want to show that $T$ is NTP$_{2}$-distal. Let
$I=\sequence{a_{i}}{i\in\left[0,1\right]}$ be indiscernible. We want
to show that $\mu=\Av_{I}|_{I}$ is NTP$_{2}$-smooth. If $I$ is
constant then $\mu$ has a unique extension (the unique realized type
of the element in $I$), so we may assume it is not. The proof is
similar to what is done in the NIP case \cite[Proposition 2.21]{Distal},
and like there it relies on several steps. 

Let $\supp\left(\mu\right)=\set{p\in S\left(I\right)}{\varphi\in p\Rightarrow\mu\left(\varphi\right)>0}$,
in other words, the set of all weakly random types over $I$. 
\begin{claim}
\label{claim:every type in the support is a limit}If $p\in\supp\left(\mu\right)$
then $p=\lim^{+}\left(\alpha\right)$ or $\lim^{-}\left(\alpha\right)$
for some $\alpha\in\left[0,1\right]$. That is, either {[}there is
some $0\leq\alpha<1$ such that for every formula $\varphi\left(x\right)$
over $I$, $\varphi\in p$ iff for some $\alpha<\beta$, for all $\alpha<\gamma<\beta$,
$\varphi\left(a_{\gamma}\right)$ holds{]} or {[}there is $0<\alpha\leq1$
such that for every formula $\varphi\left(x\right)$ over $I$, $\varphi\in p$
iff for some $\beta<\alpha$, for all $\beta<\gamma<\alpha$, $\varphi\left(a_{\gamma}\right)$
holds{]}. 
\end{claim}

\begin{proof}[Proof of Claim.]
Note that for every formula $\varphi\left(x\right)$ over $I$, $\varphi$
defines a union of intervals from $\left[0,1\right]$ which abusing
notation we will denote as $\varphi\left(I\right)$. Let $\alpha\in\bigcap_{\varphi\in p}\cl\set{\beta\in\left[0,1\right]}{\models\varphi\left(a_{\beta}\right)}$,
which exists by compactness of $\left[0,1\right]$ and the fact that
$p\in\supp\left(\mu\right)$. For every formula $\varphi\in p$, either
$\alpha$ is an isolated point of $\varphi\left(I\right)$ or $\varphi\left(I\right)$
contains an open interval to the left or right of $\alpha$. Since
$p$ is closed under finite intersection and contains $x\neq a_{\alpha}$,
we can assume that for $\varphi\in p$, $\varphi$ contains an interval
to, say, the right of $\alpha$, and as both $p$ and $\lim^{+}\left(\alpha\right)$
are complete, $p=\lim^{+}\left(\alpha\right)$. 
\end{proof}
Suppose for contradiction that $\mu$ is not NTP$_{2}$-smooth. Then
there is a formula $\varphi\left(x,d\right)$ which divides over $I$
and some extension of $\mu$ which gives it positive measure. 

Let $\Sigma=\set{\theta\in L_{x}\left(I\right)}{\varphi\left(x,d\right)\vdash\theta}$.
By Fact \ref{fact:extending measures}, $\inf\set{\mu\left(\theta\right)}{\theta\in\Sigma}$
is positive. This means that we can find $p\in\supp\left(\mu\right)$
such that $p$ contains $\Sigma$ (consider $\Sigma\cup\set{\neg\psi\left(x\right)}{\psi\in L_{x}\left(I\right),\mu\left(\psi\right)=0}$
--- this set is consistent because $\Sigma$ is closed under finite
intersections, and a finite union of sets of measure 0 is still of
measure 0 --- and let $p$ be any complete type containing it). According
to \cite[Beginning of Section 7.1]{pierrebook}, $\mu$ can be thought
of as a $\sigma$-additive Borel measure on the space of types $S_{x}\left(I\right)$
such that for every closed set $F$, $\mu\left(F\right)=\inf\set{\mu\left(D\right)}{F\subseteq D,D\text{ clopen}}$.
In particular, $\mu\left(\Sigma\right)$ is positive (where we identify
$\Sigma$ with the set of types containing it). 
\begin{claim}
For every type $p\in S_{x}\left(I\right)$, $\mu\left(\left\{ p\right\} \right)=0$.
\end{claim}

\begin{proof}[Proof of Claim.]
 Note that $I$ is not totally indiscernible, as otherwise for every
$\alpha\in\left(0,1\right)$, $I_{<\alpha}+I_{>\alpha}$ is indiscernible
over $a_{\alpha}$ and $x=a_{\alpha}$ divides over $I_{<\alpha}+I_{>\alpha}$,
contradiction. Hence, there is a formula $\chi\left(x_{0},\ldots,x_{n-1}\right)$
and a tuple $\alpha_{0}<\ldots<\alpha_{n-1}\in\left[0,1\right]$,
and $i<i+1<n$ such that $\chi\left(a_{\alpha_{0}},\ldots,a_{\alpha_{n-1}}\right)$
holds but $\chi\left(a_{\alpha_{0}},\ldots,a_{\alpha_{i+1}},a_{\alpha_{i}},\ldots,a_{\alpha_{n-1}}\right)$
does not. Let $\alpha_{i-1}<\beta<\alpha_{i}<\alpha_{i+1}<\gamma<\alpha_{i+2}$.
Then by indiscernibility $\chi\left(a_{\alpha_{0}},\ldots a_{\alpha_{i-1}},x,y,a_{\alpha_{i+2}},\ldots,a_{\alpha_{n-1}}\right)$
defines the order relation on $\left(\beta,\gamma\right)$. By indiscernibility
again, for any $\varepsilon>0$ there is a formula $\psi_{<}\left(x,y\right)$
over $\set{a_{\alpha}}{\alpha\in\left[0,\varepsilon\right]\cup\left[1-\varepsilon,1\right]}$
which defines the order relation on $\left(\varepsilon,1-\varepsilon\right)$.
Partition $\left(\varepsilon,1-\varepsilon\right)$ into finitely
many intervals $J_{k}$, each of length $\leq\varepsilon$, which
are then definable by $\psi_{k}$. Then $\mu\left(\neg\bigvee\psi_{k}\right)\leq2\varepsilon$,
and $\mu\left(\psi_{k}\right)\leq3\varepsilon$ for each $k$ (if
$\psi_{k}\left(a_{\alpha}\right)$ holds, then if $\alpha\in\left(\varepsilon,1-\varepsilon\right)$
it must be in the interval defined by $\psi_{k}$). Since $p$ is
complete, it follows by the definition of $\mu$ on closed sets that
$\mu\left(\left\{ p\right\} \right)=0$. 
\end{proof}
It follows that for any closed set $F$ with $\mu\left(F\right)>0$,
we can find infinitely many types $p_{i}\in\supp\left(\mu\right)\cap F$
(let $\varepsilon=\mu\left(F\right)$. Find $p\in\supp\left(\mu\right)\cap F$,
and a clopen set containing $p$ of measure $\leq\varepsilon/2$.
Removing this set we still have a closed set with measure at least
$\varepsilon/2$, so we can go on). 

Thus, since $\mu\left(\Sigma\right)>0$, by Claim \ref{claim:every type in the support is a limit}
we can find infinitely many $\alpha$'s in $\left(0,1\right)$ such
that $\lim^{+}\left(\alpha\right)$ or $\lim^{-}\left(\alpha\right)$
satisfy $\Sigma$. Without loss they are all of the form $\lim^{-}\left(\alpha\right)$.
Enumerate them as $\sequence{\alpha_{i}}{i<\omega}$. By definition
of $\Sigma$, for each $i<\omega$ we can find $b_{\alpha_{i}}$ such
that $\sequence{a_{\alpha}}{\alpha<\alpha_{i}}+b_{\alpha_{i}}+\sequence{a_{\alpha}}{\alpha\geq\alpha_{i}}$
is indiscernible and $\varphi\left(b_{\alpha_{i}},d\right)$ holds.
To see this, it is enough to show that $\lim^{-}\left(\alpha\right)$
is consistent with $\varphi\left(x,d\right)$. If not, then there
is a formula $\theta\left(x\right)\in\lim^{-}\left(\alpha\right)$
such that $\theta\left(x\right)\vdash\neg\varphi\left(x,d\right)$,
meaning that $\neg\theta\in\Sigma$, but then $\neg\theta\in\lim^{-}\left(\alpha\right)$.

Now find a dense indiscernible sequence $\sequence{\left(a_{i}'b_{i}'\right)}{i\in\Qq}$
which is indiscernible over $d$ and realizes the EM-type of $\sequence{\left(a_{\alpha_{i}}b_{\alpha_{i}}\right)}{i<\omega}$
over $d$. Let $I'=\sequence{a_{i}'}{i\in\Qq}$, then, for every $i\in\Qq$,
$I'_{<i}+I'_{>i}$ is indiscernible over $d$, $I'_{<i}+b_{i}'+I'_{>i}$
is indiscernible, $\varphi\left(b_{i},d\right)$ holds, and $\varphi\left(x,d\right)$
divides over $I'_{<\alpha}+I'_{>\alpha}$ --- this is a contradiction
to the assumption that the condition on the right hand side holds. 

The statement of the theorem ends with a ``Moreover'' clause which
follows as we now explain. Suppose that $\varphi\left(x_{0},\ldots,x_{n-1},y\right)$
is over $I$, $\varphi\left(x_{0},\ldots,x_{n-1},d\right)$ divides
over $I$ and that $a_{0},\ldots,a_{n-1}$ is an increasing tuple
from $I_{1}+b+I_{2}$ and that $\varphi\left(a_{0},\ldots,a_{n-1},d\right)$
holds. Suppose that $a_{i}=b$ (if $b$ does not belong to $a_{0},\ldots,a_{n-1}$
then we already get a contradiction since dividing formulas are not
satisfiable in the base). Consider the formula $\psi\left(x\right)=\varphi\left(a_{0},\ldots,a_{i-1},x,a_{i+1},\ldots,a_{n-1},d\right)$
(which has some more parameters from $I$). Since $\varphi\left(x_{0},\ldots,x_{n-1},d\right)$
divides over $I$, it follows that $\psi\left(x\right)$ divides over
$I'$ where $I'=I_{1}'+I_{2}'$ and $I_{1}'$ is an end-segment of
$I_{1}$ and $I_{2}'$ an initial segment of $I_{2}$ disjoint from
the parameters of $\psi$ (with the same sequence of conjugates of
$d$ augmented with the parameters of $\psi$ witnessing this). By
the previous result, we get a contradiction. 
\end{proof}
\begin{prop}
The theory $T$ is distal iff it is NTP$_{2}$-distal and NIP. 
\end{prop}

\begin{proof}
Suppose that $T$ is NTP$_{2}$-distal and NIP. By Fact \ref{fact:Distal iff all averages are smooth},
we have to show that for any model $M$ and $I\subseteq M$ indiscernible
indexed by $\left[0,1\right]$, $\Av_{I}|_{M}$ is smooth. By Proposition
\ref{prop:in NIP smooth =00003D NTP2 smooth}, it is enough to show
that $\Av_{I}|_{M}$ is NTP$_{2}$-smooth. Suppose that $\varphi\left(x,c\right)$
divides over $M$ and that $\mu\left(\varphi\left(x,c\right)\right)>0$
for some $\mu$ extending $\Av_{I}|_{M}$. In particular $\mu$ extends
$\Av_{I}|_{I}$ and $\varphi\left(x,c\right)$ divides over $I$,
so the latter is not NTP$_{2}$-smooth --- a contradiction to the
definition of NTP$_{2}$-distality. 

On the other hand, if $T$ is distal, then it is NIP (see after Definition
\ref{def:distal}). If the reader believes that Fact \ref{fact:Distal iff all averages are smooth}
is also true over $I$ (which is not stated but follows from the proof
of this fact), then there is a unique global extension of $\Av_{I}|_{I}$,
namely $\Av_{I}|_{\C}$, which is also finitely satisfiable in $I$
(if $\Av_{I}\left(\varphi\left(x,c\right)\right)>0$ then $\varphi\left(I,c\right)\neq\emptyset$)
so does not divide over $I$ (see Definition \ref{def:NTP2 smooth}).
In particular every extension of $\Av_{I}|_{I}$ does not divide over
$I$ so it is NTP$_{2}$-smooth. 

For the skeptic reader, we also give an alternative proof using Theorem
\ref{thm:NTP2 distal in terms of sequence}: if $I=I_{1}+I_{2}$ is
an indiscernible sequence with $I_{1}$ endless and $I_{2}$ with
no first element and $I$ is $d$-indiscernible, $I_{1}+b+I_{2}$
is indiscernible then by distality $I_{1}+b+I_{2}$ is indiscernible
over $d$. Hence, if $\varphi\left(x,d\right)$ divides over $I$
then $\neg\varphi\left(b,d\right)$ holds (because if $\varphi\left(b,d\right)$
holds then $\varphi\left(a,d\right)$ holds for all $a\in I$ and
in particular $\varphi\left(x,d\right)$ is satisfiable in $I$ so
cannot divide over $I$). 
\end{proof}
\begin{question}
Is being NTP$_{2}$-distal preserved under forgetting parameters?
In other words, is it true that if $A\subseteq\C$ and $T_{A}=\Th\left(\C_{A}\right)$
is NTP$_{2}$-distal then so is $T$?
\end{question}

\begin{rem}
Being NTP$_{2}$-distal is preserved under naming parameters. To see
this, first note that being NTP$_{2}$ is stable under naming and
forgetting parameters. Suppose that $T$ is NTP$_{2}$-distal, and
fix some $A$. By Theorem \ref{thm:NTP2 distal in terms of sequence},
it is enough to show that if $I_{1}+b+I_{2}$ is $A$-indiscernible,
$I=I_{1}+I_{2}$ is $Ad$-indiscernible and $\varphi\left(x,d\right)$
divides over $AI$ (where $\varphi$ is over $A$), then $\neg\varphi\left(b,d\right)$
holds (where $I_{1}$ has no last element and $I_{2}$ has no first
element). Now incorporate $A$ (or even just the parameters from $A$
appearing in $\varphi$) into $I_{1}+b+I_{2}$ by concatenating these
elements in the end of every tuple from the sequence and do the same
with $d$. Then we are in the situation of Theorem \ref{thm:NTP2 distal in terms of sequence},
but over $\emptyset$, so we are done.
\end{rem}

\begin{example}
\label{exa:The-ordered-random}The ordered random graph is NTP$_{2}$-distal.
The ordered random graph is the model companion of the theory of ordered
graphs in the language $\left\{ R,<\right\} $ where the order and
the graph are independent. The restriction to the order part is just
DLO and hence distal, and the restriction to the graph $R$ is the
random graph. Note the following easy facts:
\begin{itemize}
\item If $p_{<}\left(x\right)$ and $p_{R}\left(x\right)$ are (partial)
types over any set $A$ in $\left\{ <\right\} $, $\left\{ R\right\} $
respectively, whose restrictions to $\left\{ =\right\} $ is a complete
type: $p_{<}\restriction\left\{ =\right\} =p_{R}\restriction\left\{ =\right\} \in S_{x}^{=}\left(A\right)$,
then their union $p$ is a consistent type over $A$ (after taking
completions over $A$, this union determines a unique ordered graph
with universe $Ax$, which can be embedded into $\C$ in a way that
fixes $A$) and if $p_{<},p_{R}$ are complete, so is $p$. Here $x$
is any (finite) tuple of variables. 
\item It follows that if $p\left(x\right)$ is a complete type over $A$
which divides over some $B$ then either its restriction to $\left\{ <\right\} $
or its restriction to $\left\{ R\right\} $ divides over $B$. 
\item Call a complete type $p$ strongly non-algebraic if every realization
of it is disjoint to its domain. Since strongly non-algebraic types
in the random graph do not divide (over $\emptyset$ so over every
set), it follows that in such a case, $p_{<}=p\restriction\left\{ <\right\} $
must divide. 
\item Suppose we are in the situation of Theorem \ref{thm:NTP2 distal in terms of sequence},
i.e., with $A=\emptyset$, $I=I_{1}+I_{2}$, $I_{1}+b+I_{2}$ is indiscernible
over $A$, $I$ is $Ad$-indiscernible, $\varphi\left(x,y\right)$
is over $A$, $\varphi\left(x,d\right)$ divides over $AI$, but towards
a contradiction, assume that $\varphi\left(b,d\right)$ holds. Now
we enlarge $A$, shrink every tuple from $I$ and change $\varphi$
accordingly to get that  the intersection of any two tuples is empty
($A$ will be the elements common to all/any two tuples in $I$ or
even $I_{1}+b+I_{2}$). Note that everything is preserved by doing
that. Then: if $p\left(x\right)=\tp\left(b/dA\right)$ is strongly
non-algebraic then $p_{<}$ must divide over $I$, but as DLO is distal
$I_{1}+b+I_{2}$ is indiscernible over $Ad$ in $\left\{ <\right\} $
and hence $p_{<}\left(a\right)$ holds for all $a\in I$ and in particular
$p_{<}$ is finitely satisfiable in $I$ so cannot divide over it. 
\item It follows that in such a situation $p\left(x\right)$ is not strongly
non-algebraic, so one of the points $b'$ in the tuple $b$ is equal
to some $d'\in d$ (it is impossible that $b'\in A$). As $I_{1}+b+I_{2}$
is $Ad$-indiscernible in $\left\{ <\right\} $, it follows that in
the coordinate of $b'$, $I$ must be constant, contradiction to the
assumption that the intersection of any two tuples from $I$ is empty.
\end{itemize}
\end{example}

\begin{example}
\label{exa:The-random-tournament}The random tournament is NTP$_{2}$-distal.
This theory $T$ is the model companion of the theory of tournaments:
it is a universal theory in the language $\left\{ R\right\} $ where
$R$ is a binary relation whose only axiom is $\forall xyR\left(x,y\right)\leftrightarrow\neg R\left(y,x\right)$.
The theory $T$ is supersimple of $U$-rank 1. In other words, if
$\varphi\left(x,a\right)$ forks over $A$ where $x$ is a singleton
then $\varphi$ is algebraic (i.e., $\varphi\vdash\bigvee_{i<k}x=c_{i}$).
In fact, any union of strongly non-algebraic complete types $p_{i}\left(x\right)$
extending a complete type $p\left(x\right)$ over $A$ with $x$ any
finite tuple over $A_{i}$ for $i<\omega$ such that $A_{i}\cap A_{j}=A$
for all $i<j$ is consistent (consider the tournament on $\bigcup_{i}A_{i}x$
given by all the $p_{i}$'s and embed it into $\C$ in a way that
fixes $\bigcup_{i}A_{i}$). 

As in Example \ref{exa:The-ordered-random}, to show that $T$ is
NTP$_{2}$-distal, we assume that $I=I_{1}+I_{2}$ is $Ad$-indiscernible,
$I_{1}+b+I_{2}$ is $A$-indiscernible, the intersection of any two
distinct tuples from $I$ is empty and $\varphi\left(x,d\right)$
divides over $AI$. Suppose that $\varphi\left(b,d\right)$ holds.
It follows from the previous paragraph that $p\left(x\right)=\tp\left(b/Ad\right)$
is not strongly non-algebraic. Since $b\cap A=\emptyset$, there must
be some $b'\in b$ and $d'\in d$ such that $b'=d'$. Suppose that
$a_{1}\in I_{1}$ and $a_{1}'\in a_{1}$ is in the same coordinate
as $b'$, and $R\left(a_{1}',b'\right)$ holds, then it must be that
$R\left(b',a_{2}'\right)$ holds for any $a_{2}\in I_{2}$. But as
$I$ is indiscernible over $d$, it follows that $R\left(a_{2}',b'\right)$
holds as well --- contradiction to the definition of a tournament.
Otherwise, $R\left(b',a_{1}'\right)$ holds and apply the symmetric
argument. 
\end{example}

Example \ref{exa:The-random-tournament} shows that as opposed to
the NIP case, where no distal stable theory exists, it is true that
there are NTP$_{2}$-distal simple theories. This raises questions
about this definition, which leaves room for future investigations. 
\begin{problem}
Find more examples of NTP$_{2}$ -distal theories. Some natural candidates
include the theory of bounded PRC fields, see \cite{MR3564381}. 
\end{problem}

\section{\label{sec:finite alternation}On $\omega$-resilience and a finite
alternation theorem}

Recall that $T$ is called \emph{resilient} if whenever $\sequence{a_{i}}{i\in\Zz}$
is indiscernible, and $\varphi\left(x,a_{0}\right)$ divides over
$a_{\neq0}$, then $\sequence{\varphi\left(x,a_{i}\right)}{i\in\Zz}$
is inconsistent. This notion was introduced in \cite{MR3226015}.
All NIP and simple theories are resilient, and all resilient theories
are NTP$_{2}$. It is conjectured that NTP$_{2}$ theories are all
resilient. 
\begin{defn}
\label{def:omega-resilient}Say that $T$ is \emph{$\omega$-resilient}
if whenever $\sequence{a_{i}}{i<\omega}$ and $\sequence{b_{i}}{i<\omega}$
are such that:
\begin{enumerate}
\item Both $\sequence{a_{i}}{i<\omega}$ and $\sequence{b_{i}}{i<\omega}$
are indiscernible.
\item For every $k<\omega$, $\sequence{a_{i}}{i\leq k}+\left\langle b_{k}\right\rangle +\sequence{a_{i}}{k<i<\omega}$
is indiscernible (in particular, $b_{i}$ and $a_{j}$ are all tuples
of the same length).
\end{enumerate}
Then for every formula $\varphi\left(x,y\right)$, if $\set{\varphi\left(x,a_{i}\right)}{i<\omega}$
is consistent, then so is $\set{\varphi\left(x,b_{i}\right)}{i<\omega}$. 
\end{defn}

\begin{rem}
\label{rem:Get an array}
\begin{enumerate}
\item In Definition \ref{def:omega-resilient}, using Ramsey we could replace
(1) by asking that $\sequence{a_{i}b_{i}}{i<\omega}$ is indiscernible.
\item If $T$ is resilient then $T$ is $\omega$-resilient. Why? Suppose
that $\sequence{a_{i}b_{i}}{i<\omega}$ is indiscernible and as in
the definition and $\set{\varphi\left(x,a_{i}\right)}{i<\omega}$
is consistent, but $\set{\varphi\left(x,b_{i}\right)}{i<\omega}$
is inconsistent. Then the latter is also $k$-inconsistent for some
$k<\omega$. Then for any $N<\omega$, $\sequence{b_{N+i}}{i<N}$
is a sequence of tuples realizing $\tp\left(a_{N}/\left(a_{<N}\cup a_{>N+N}\right)\right)$
such that $\set{\varphi\left(x,b_{N+i}\right)}{i<N}$ is $k$-inconsistent.
Hence by compactness we find a sequence contradicting resilience. 
\item \cite[Proposition 4.5]{MR3226015} says that if $\sequence{a_{i}}{i<\omega}$
and $\sequence{b_{i}}{i<\omega}$ are indiscernible sequences in the
sort of $y$ such that for every $\varphi\left(x,y\right)$, if $\set{\varphi\left(x,a_{i}\right)}{i<\omega}$
is consistent then so is $\set{\varphi\left(x,b_{i}\right)}{i<\omega}$
then there is an array (with mutually indiscernible rows) $\sequence{c_{i,j}}{i,j<\omega}$
such that every row $\bar{c}_{i}\equiv\sequence{b_{i}}{i<\omega}$
and for every $\eta:\omega\to\omega$, $\sequence{c_{i,\eta\left(i\right)}}{i<\omega}\equiv\sequence{a_{i}}{i<\omega}$.
Thus, if $T$ is $\omega$-resilient and $\sequence{a_{i}b_{i}}{i<\omega}$
are as in the definition then we can find such an array. 
\end{enumerate}
\end{rem}

\begin{question}
Does $\omega$-resilience imply NTP$_{2}$? Does NTP$_{2}$ imply
$\omega$-resilience?
\end{question}

The following is the main theorem for this section.
\begin{thm}
\label{thm:finite alternation in NTP2}Suppose that $T$ is $\omega$-resilient
and NTP$_{2}$. Suppose that $\sequence{a_{i}}{i<\omega}$ is an indiscernible
sequence and that $\varphi\left(x,b\right)$ divides over $\sequence{a_{2i}}{i<\omega}$.
Then for all but finitely many $i$'s, $\neg\varphi\left(a_{i},b\right)$
holds. 
\end{thm}

Note that this is true when $T$ is simple: if not, by Ramsey (and
pigeonhole), we may assume that $\sequence{a_{2i}a_{2i+1}}{i<\omega}$
is indiscernible over $b$ and $\varphi\left(a_{2i+1},b\right)$ holds
for all $i<\omega$ (note that $\varphi\left(a_{2i},b\right)$ never
holds by assumption). We now extend the sequence $\sequence{a_{i}}{i<\omega}$
to have order type $\omega+\omega$, and let $I_{1}=\sequence{a_{2i}}{i<\omega}$
and $I_{2}=\sequence{a_{2i+1}}{\omega\leq i<\omega+\omega}$. Then
$\varphi\left(x,b\right)$ divides over $I_{1}$ so by symmetry $b\nind_{I_{1}}a_{2i+1}$
for every $\omega\leq i$ as witnessed by some formula $\psi\left(x,a_{2i+1}\right)\in\tp\left(b/a_{2i+1}I_{1}\right)$
over $I_{1}$ (by indiscernibility it is the same formula). However
$I_{2}$ is a reverse Morley sequence over $I_{1}$ (in the sense
that $a_{2i+1}\ind_{I_{1}}a_{>2i+1}$ by finite satisfiability) so
by Kim's Lemma (see e.g., \cite[Proposition 2.1]{KimForking}), $\set{\psi\left(x,a_{2i+1}\right)}{\omega\leq i<\omega}$
is inconsistent --- this is a contradiction because $b$ realizes
it by indiscernibility. 

Note also that the conclusion of Theorem \ref{thm:finite alternation in NTP2}
is true if $T$ is NIP (this is trivial) and when $T$ is NTP$_{2}$-distal.
The latter follows from Theorem \ref{thm:NTP2 distal in terms of sequence}
(we leave the details to the reader; the argument starts the same
way as in the simple case, taking a sequence of order type $\Qq+1+\Qq$).

Before the proof let us recall a simple criterion for having TP$_{2}$.
\begin{fact}
\label{fact:Criterion for TP2}\cite[Lemma 2.24]{22-Levi2016}Suppose
that $A$ is some infinite set in $\C$ and $\varphi\left(x,y\right)$
is a formula such that for some $k<\omega$ , for every sequence $\sequence{A_{i}}{i<\omega}$
of pairwise disjoint subsets of $A$, there are $\sequence{b_{i}}{i<\omega}$
such that $A_{i}\subseteq\varphi\left(\C,b_{i}\right)$ and $\set{\varphi\left(x,b_{i}\right)}{i<\omega}$
is $k$-inconsistent. Then $T$ has TP$_{2}$. 
\end{fact}

\begin{cor}
\label{cor:better criterion for TP2}Suppose that $\sequence{a_{n,m}}{n,m<\omega}$
is an indiscernible sequence, ordered lexicographically and that $\varphi\left(x,y\right)$
is a formula such that for some $k<\omega$, there is a sequence $\sequence{b_{m}}{m<\omega}$
such that for every $n,m<\omega$, $\varphi\left(a_{n,m},b_{m}\right)$
holds and $\set{\varphi\left(x,b_{m}\right)}{m<\omega}$ is $k$-inconsistent.
Then $T$ has TP$_{2}$. 
\end{cor}

\begin{proof}[Proof of Corollary]
 We show that $A=\set{a_{n,0}}{n<\omega}$ (the $0$'th column) has
the property of Fact \ref{fact:Criterion for TP2}. Suppose that $\sequence{A_{m}}{m<\omega}$
is a sequence of disjoint subsets of $A$. For each $m<\omega$, let
$\bar{a}_{m}$ enumerate $\sequence{a_{n,0}}{n<\omega,a_{n,0}\in A_{m}}$
and let $\bar{a}_{m}'$ enumerate $\sequence{a_{n,m}}{n<\omega,a_{n,0}\in A_{m}}$
(both in increasing order according to $n$). By indiscernibility,
$\sequence{\bar{a}_{m}'}{m<\omega}\equiv\sequence{\bar{a}_{m}}{m<\omega}$.
Hence there are $b_{m}'$ for $m<\omega$ such that $\sequence{b_{m}\bar{a}_{m}'}{m<\omega}\equiv\sequence{b_{m}'\bar{a}_{m}}{m<\omega}$.
These will satisfy the conditions of Fact \ref{fact:Criterion for TP2}. 
\end{proof}
Finally we are ready to prove the theorem. 
\begin{proof}[Proof of Theorem \ref{thm:finite alternation in NTP2}]
For notational simplicity, write $a_{i}$ for $a_{2i}$ and $c_{i}$
for $a_{2i+1}$. 

Assume towards contradiction that the conclusion does not hold. We
may assume, restricting to a subsequence, that for every $i<\omega$,
$\varphi\left(c_{i},b\right)$ holds (note that $\neg\varphi\left(a_{i},b\right)$
holds for all $i<\omega$ by the definition of dividing). Applying
Ramsey and compactness, we may assume that $\sequence{a_{i}c_{i}}{i<\omega}$
is indiscernible over $b$. Next, we may assume its order type is
$\omega\x\omega$ (ordered lexicographically), so we have $\sequence{a_{n,m}c_{n,m}}{n,m<\omega}$.
Let $\bar{a}=\sequence{a_{n,m}}{n,m<\omega}$ and similarly define
$\bar{c}$. 

For $k<\omega$, let $\bar{a}_{k}=\sequence{a_{n,k}}{n<\omega}$ (the
$k$'th column in $\bar{a}$), and similarly, let $\bar{c}_{k}=\sequence{c_{n,k}}{n<\omega}$.
Then for all $N<\omega$, $\sequence{\bar{a}_{k}}{k\leq N}+\left\langle \bar{c}_{N}\right\rangle +\sequence{\bar{a}_{k}}{N<k}$
is indiscernible. 

Let $\bar{b}=\sequence{b_{j}}{i<\omega}$ witness that $\varphi\left(x,b\right)$
divides and even $r$-divides over $\bar{a}$. For each $k<\omega$,
find $\bar{c}^{k}$ such that $b_{k}\bar{c}^{k}\equiv_{\bar{a}}b\bar{c}_{k}$.
In particular, $\sequence{\bar{a}_{k}}{k\leq N}+\left\langle \bar{c}^{N}\right\rangle +\sequence{\bar{a}_{k}}{N<k}$
is indiscernible for all $N<\omega$. Note that $\bar{c}^{k}b_{k}\equiv\bar{c}_{0}b$
for all $k<\omega$ because $\bar{c}$ is $b$-indiscernible. Using
Ramsey and compactness again, find an indiscernible sequence $\sequence{\bar{e}_{k}\bar{f}_{k}}{k<\omega}$
realizing the EM-type of $\sequence{\bar{a}_{k}\bar{c}^{k}}{k<\omega}$. 

We still have that $\sequence{\bar{e}_{k}}{k\leq N}+\left\langle \bar{f}_{N}\right\rangle +\sequence{\bar{e}_{k}}{k>N}$
is indiscernible for all $N$. By $\omega$-resilience and Remark
\ref{rem:Get an array}(3), there is an array $\sequence{\bar{h}_{n,m}}{n,m<\omega}$
such that for every $n<\omega$, $\sequence{\bar{h}_{n,m}}{m<\omega}\equiv\sequence{\bar{f}_{k}}{k<\omega}$
and for every $\eta:\omega\to\omega$, $\sequence{\bar{h}_{n,\eta\left(n\right)}}{n<\omega}\equiv\sequence{\bar{e}_{k}}{k<\omega}\equiv\sequence{\bar{a}_{k}}{k<\omega}$. 

For each $n,m$ we can find $b_{n,m}$ such that $\bar{h}_{n,m}b_{n,m}\equiv\bar{c}_{0}b$
and $\set{\varphi\left(x,b_{n,m}\right)}{m<\omega}$ is $r$-inconsistent
(because this is a closed condition which holds for the sequence $\sequence{\bar{c}^{k}}{k<\omega}$).
By Ramsey and compactness we may assume that the whole array $\sequence{\bar{h}_{n,m}b_{n,m}}{n,m<\omega}$
is indiscernible in the sense that the rows are mutually indiscernible,
and even that the sequence of rows 
\[
\sequence{\sequence{\bar{h}_{n,m}b_{n,m}}{m<\omega}}{n<\omega}
\]
 is itself indiscernible.

By NTP$_{2}$, there is some $\eta:\omega\to\omega$ such that $\set{\varphi\left(x,b_{n,\eta\left(n\right)}\right)}{n<\omega}$
is inconsistent, so by indiscernibility, this is true for $\eta$
being constantly $0$ and hence it is $l$-inconsistent for some $l$.
As the sequence $\sequence{\bar{h}_{k,0}}{k<\omega}\equiv\sequence{\bar{a}_{k}}{k<\omega}$,
we can find $b_{k}'$ such that $\set{\varphi\left(x,b_{k}'\right)}{k<\omega}$
is $l$-inconsistent and $\varphi\left(a_{n,k},b_{k}'\right)$ holds
for all $n<\omega$. 

However this contradicts NTP$_{2}$ by Corollary \ref{cor:better criterion for TP2}.
\end{proof}
\begin{cor}
\label{cor:dividing IP} Suppose that $T$ is $\omega$-resilient
and NTP$_{2}$. Then the following is impossible: 

There exists an infinite set $A$, a formula $\varphi\left(x,y\right)$
and some $k<\omega$ such that for every subset $s\subseteq A$, there
is some $b_{s}$ such that $\varphi\left(x,b_{s}\right)$ $k$-divides
over $A\backslash s$ and $s\subseteq\varphi\left(\C,b_{s}\right)$.
\end{cor}

\begin{proof}
Suppose that there are such a set $A$, a formula $\varphi\left(x,y\right)$
and $k$. Without loss of generality, $A$ is countable. Enumerate
$A$ as $\bar{a}=\sequence{a_{i}}{i<\omega}$. Let $\bar{a}'=\sequence{a_{i}'}{i<\omega}$
be an indiscernible sequence realizing the EM-type of $\bar{a}$.
Then there is some $b$ such that $\varphi\left(x,b\right)$ $k$-divides
over $\set{a_{2i}'}{i<\omega}$ and $\varphi\left(a_{2i+1},b\right)$
holds for all $i<\omega$. This contradicts Theorem \ref{thm:finite alternation in NTP2}. 
\end{proof}

We end this section with an open problem which we find extremely nice.
\begin{question}
Suppose that $T$ is NTP$_{2}$, that $\sequence{b_{i}}{i<\omega}$
is a Morley sequence (in the sense of non-forking) over a model $M$,
and that $\varphi\left(x,b_{0}\right)$ divides over $M$. Is it true
that $\set{\varphi\left(x,b_{2i}\right)\wedge\neg\varphi\left(x,b_{2i+1}\right)}{i<\omega}$
is inconsistent? 
\end{question}

\section{\label{sec:On-singular-local}On singular local character in NIP}

Here we prove a theorem on what we call singular local character in
the setting of NIP. The idea is that local character for non-forking
fails for general NIP theories, but we can still recover some version
of it over sets of singular cardinality. This is a new result for
NIP, but of course it holds for simple theories by the definition
of simplicity. This raises a natural question on NTP$_{2}$ (see Question
\ref{que:generalizing theorem to NTP2}). 
\begin{thm}
\label{thm:singular local character over one model}Suppose that $T$
is NIP. Suppose that  $A\subseteq\C$ is a small set of cardinality
$\mu$ where, $\left|T\right|<\cof\left(\mu\right)=\kappa<\mu$. Then
for every (finitary) type $p\left(x\right)\in S\left(A\right)$ there
is some $B\subseteq A$ of cardinality $<\mu$ such that $p$ does
not divide over $B$. 

If $A$ is a model $M$, then there is some model $N\prec M$ with
$\left|N\right|<\mu$ such that $p$ is finitely satisfiable in $N$.
\end{thm}

\begin{proof}
We start with the first statement. We use the same ideas as in \cite[Theorem 2.12]{Sh900}.
Write $A=\bigcup\set{A_{i}}{i<\kappa}$ where $\left|A_{i}\right|<\mu$
for all $i<\kappa$ and $A_{i}\subseteq A_{j}$ for $i\le j$. 

Fix some $d\models p$. Let $X$ be the set of sequences $\bar{c}=\sequence{\left(c_{\alpha}^{0},c_{\alpha}^{1}\right)}{\alpha<\gamma_{\bar{c}}}$
with $\gamma_{\bar{c}}\leq\left|T\right|^{+}$ such that, letting
$c_{\alpha}=\left(c_{\alpha}^{0},c_{\alpha}^{1}\right)$:
\begin{enumerate}
\item For every $\alpha<\gamma_{\bar{c}}$, $c_{\alpha}^{0}\equiv_{A\bar{c}_{<\alpha}}c_{\alpha}^{1}$.
\item $c_{\alpha}^{0}d\not\equiv_{A\bar{c}_{<\alpha}}c_{\alpha}^{1}d$.
\item For some $i<\kappa$, $\tp\left(c_{\alpha}/A\bar{c}_{<\alpha}\right)$
is finitely satisfiable in $A_{i}$.
\end{enumerate}
We try to construct a maximal element in $X$ of length $<\left|T\right|^{+}$,
i.e., one that cannot be extended to a longer sequence. Suppose we
cannot, i.e., there is $\bar{c}\in X$ with $\gamma_{\bar{c}}=\left|T\right|^{+}$.
For each $\alpha<\left|T\right|^{+}$, there is a formula $\varphi_{\alpha}\left(x,y,z\right)$
over $\emptyset$ and $b_{\alpha}\in A\bar{c}_{<\alpha}$ such that
$\varphi_{\alpha}\left(d,c_{\alpha}^{0},b_{\alpha}\right)\land\neg\varphi_{\alpha}\left(d,c_{\alpha}^{1},b_{\alpha}\right)$.
Extracting we may assume that $\varphi_{\alpha}=\varphi$. By Fodor's
lemma (applied to the function taking any $\alpha<\left|T\right|^{+}$
to the minimal $\varepsilon<\alpha$ such that $b_{\alpha}\in A\bar{c}_{\leq\varepsilon}$),
for some $\beta<\left|T\right|^{+}$ there is a stationary subset
$S$ of $\left|T\right|^{+}\backslash\beta$ such that for all $\alpha\in S$,
$b_{\alpha}\in A\bar{c}_{<\beta}$. 

Now note that for any $\eta:S\to2$, $\tp\left(\sequence{c_{\alpha}^{\eta\left(\alpha\right)}}{\alpha\in S}/A\bar{c}_{<\beta}\right)$
is constant regardless of $\eta$. Indeed, prove by induction that
this is true for all finite subsets of $S$. Given $\alpha_{0}<\cdots<\alpha_{n}$
from $S$ and $\eta:\set{\alpha_{i}}{i\leq n}\to2$, let $\eta'\left(\alpha_{i}\right)=\eta\left(\alpha_{i}\right)$
for all $i<n$ and $\eta'\left(\alpha_{n}\right)=0$. Then $\tp\left(\sequence{c_{\alpha_{i}}^{\eta\left(\alpha_{i}\right)}}{i\leq n}/A\bar{c}_{<\beta}\right)=\tp\left(\sequence{c_{\alpha_{i}}^{\eta'\left(\alpha_{i}\right)}}{i\leq n}/A\bar{c}_{<\beta}\right)$
by property (1) in the definition of $X$, so we may assume that $\eta\left(\alpha_{n}\right)=0$.
Now, if $\psi\left(c_{\alpha_{0}}^{\eta\left(\alpha_{0}\right)},\ldots,c_{\alpha_{n-1}}^{\eta\left(\alpha_{n-1}\right)},c_{\alpha_{n}}^{0}\right)$
and $\neg\psi\left(c_{\alpha_{0}}^{0},\ldots,c_{\alpha_{n-1}}^{0},c_{\alpha_{n}}^{0}\right)$
hold for $\psi$ over $A\bar{c}_{<\beta}$, then, by property (3)
in the definition of $X$, there is $d\in A$ such that $\psi\left(c_{\alpha_{0}}^{\eta\left(\alpha_{0}\right)},\ldots,c_{\alpha_{n-1}}^{\eta\left(\alpha_{n-1}\right)},d\right)$
and $\neg\psi\left(c_{\alpha_{0}}^{0},\ldots,c_{\alpha_{n-1}}^{0},d\right)$
hold, contradicting the induction hypothesis.

Hence we get that for any subset $s\subseteq S$, $\set{\varphi\left(x,c_{\alpha}^{0},b_{\alpha}\right)}{\alpha\in s}\cup\set{\neg\varphi\left(x,c_{\alpha}^{0},b_{\alpha}\right)}{\alpha\notin s}$
is consistent. But that clearly gives us IP for the formula $\varphi$. 

Let $\bar{c}\in X$ be maximal of length $<\left|T\right|^{+}$. 

We get that $r\left(x\right)=\tp\left(d/A\bar{c}\right)$ is weakly
orthogonal to any type $q\left(y\right)\in S\left(A\bar{c}\right)$
which is finitely satisfiable in some $A_{i}$ for $i<\kappa$, in
the sense that $q\cup r$ implies a complete type over $A\bar{c}$:
assume $e_{1},e_{2}\models q$ and $de_{1}\not\equiv_{A\bar{c}}de_{2}$.
Take a global finitely satisfiable in $A_{i}$ extension $q'$ of
$q$, and let $e'\models q'|A\bar{c}de_{1}e_{2}$. Then for one of
$i=1,2$, $de'\not\equiv_{A\bar{c}}de_{i}$. Rename $e_{1}=e_{i}$
and $e_{2}=e'$. Thus, $\tp\left(e_{2}/A\bar{c}de_{1}\right)$ is
finitely satisfiable in $A_{i}$ and in particular $\tp\left(e_{1}e_{2}/A\bar{c}\right)$
is finitely satisfiable in $A_{i}$, contradicting the maximality
of $\bar{c}$. 

Fix some formula $\varphi\left(x,y\right)$ over $A\bar{c}$. For
$i<\kappa$, let $Y_{i}$ be the set of all $q\left(y\right)\in S\left(A\bar{c}\right)$
finitely satisfiable in $A_{i}$ such that $r\left(x\right)\cup q\left(y\right)\models\varphi\left(x,y\right)$.
For each $q\in Y_{i}$ there is a formula $\zeta_{q}\in q$ and a
formula $\psi_{q}\in r$ such that $\psi_{q}\wedge\zeta_{q}\vdash\varphi$.
This implies that $Y_{i}$ is open in the space of all types over
$A\bar{c}$ finitely satisfiable in $A_{i}$, which is itself closed
in the space of all types over $A\bar{c}$. By weak orthogonality
its complement is precisely those types $q\in S\left(A\bar{c}\right)$
finitely satisfiable in $A_{i}$ such that $r\left(x\right)\cup q\left(y\right)\vdash\neg\varphi\left(x,y\right)$
so it is also open there. Thus, $Y_{i}$ is clopen in this space,
so compact. By compactness there are formulas $\zeta\left(y\right)$
and $\psi\in r$ such that $\psi\wedge\zeta\vdash\varphi$ and $\zeta$
covers $Y_{i}$. Let $E_{i}\subseteq A$ be the set of parameters
from $A$ appearing in all these formulas $\psi$ when we run over
all formulas $\varphi\left(x,y\right)$. Then $\left|E_{i}\right|\leq\left|T\right|$
and $\tp\left(d/E_{i}\bar{c}\right)\vdash\tp\left(d/A_{i}\bar{c}\right)$
because every realized type of a tuple in $A_{i}$ is in $Y_{i}$.

Note that $\tp\left(\bar{c}/A\right)$ is finitely satisfiable in
some $A_{i_{0}}$ for $i_{0}<\kappa$, because we assumed that $\kappa\geq\left|T\right|^{+}$.
Let $E=\bigcup_{i<\kappa}E_{i}$. Then $\left|E\right|<\mu$ and $p=\tp\left(d/A\right)$
does not divide over $EA_{i_{0}}\subseteq A$: suppose that $\varphi\left(x,f\right)\in p$
divides over $EA_{i_{0}}$ where $f\in A_{i}$ for some $i<\kappa$
and $\varphi$ is over $EA_{i_{0}}$. Note that as $\bar{c}\ind_{A_{i_{0}}}^{fs}Ef$
(i.e., $\tp\left(\bar{c}/A_{i_{0}}Ef\right)$ is finitely satisfiable
in $A_{i_{0}}$), $\varphi\left(x,f\right)$ also divides over $EA_{i_{0}}\bar{c}$
(any indiscernible sequence $I$ starting with $f$ over $EA_{i_{0}}$
can be moved by an automorphism so that it is also indiscernible over
$EA_{i_{0}}\bar{c}$ by ensuring that $\bar{c}\ind_{A_{i_{0}}}^{fs}EI$).
Now, $\tp\left(d/E_{i}\bar{c}\right)\vdash\varphi\left(x,f\right)$
so some formula $\psi\left(x,e,\bar{c}\right)\in\tp\left(d/E_{i}\bar{c}\right)$
divides over $EA_{i_{0}}\bar{c}$ and in particular over $E_{i}\bar{c}$
which is impossible. 

Now suppose that $A=M$ is a model. Find a model $N\prec M$ containing
$EA_{i_{0}}$ of size $\leq\left|T\right|+\left|A_{i_{0}}E\right|$.
We claim that $p$ is finitely satisfiable in $N$. Suppose that
$\varphi\left(x,f\right)\in p$. Let $\psi\left(x,e,\bar{c}\right)\in\tp\left(d/E\bar{c}\right)$
be such that $\psi\left(x,e,\bar{c}\right)\vdash\varphi\left(x,f\right)$.
Then 
\[
\C\models\exists x\psi\left(x,e,\bar{c}\right)\land\forall x\left(\psi\left(x,e,\bar{c}\right)\to\varphi\left(x,f\right)\right).
\]
 As $\bar{c}\ind_{A_{i_{0}}}^{fs}Ef$, there is some $\bar{c}'\in N$
such that 
\[
\C\models\exists x\psi\left(x,e,\bar{c}'\right)\land\forall x\left(\psi\left(x,e,\bar{c}'\right)\to\varphi\left(x,f\right)\right),
\]
and as $N\prec\C$, $N\models\exists x\psi\left(x,e,\bar{c}'\right)$,
so there is some $d'\in N$ such that $\psi\left(d',e,\bar{c}'\right)$
holds. Hence, $\varphi\left(d',f\right)$ holds as well.
\end{proof}
The following question seems natural. 
\begin{question}
\label{que:generalizing theorem to NTP2}Is Theorem \ref{thm:singular local character over one model}
true for NTP$_{2}$? Namely, suppose that $T$ is NTP$_{2}$, that
$M\models T$, $\left|T\right|<\cof\left(\left|M\right|\right)<\left|M\right|$
and that $p\left(x\right)\in S\left(M\right)$. Is there some $N\prec M$
over which $p$ does not fork? The same question can be asked with
$M$ being a set (and forking replaced by dividing; note that the
version over sets imply the version over models since forking equals
dividing over models in NTP$_{2}$). 
\end{question}

We now move to global types.

\begin{lem}
\label{lem:countable set outside of dom(p)}Suppose that $T$ is NIP
and that $p\left(x\right)\in S\left(\C\right)$ is a global type which
is finitely satisfiable in a model $M$ such that $\mu=\left|M\right|$
is singular with $\left|T\right|<\kappa=\cof\left(\mu\right)<\mu$.
Then for any countable set $A\subseteq\C$ there there is some $M_{0}\prec M$
of size $<\mu$ such that $p\restriction\left(A\cup M\right)$ is
finitely satisfiable in $M_{0}$. 
\end{lem}

Note that this implies the second part of Theorem \ref{thm:singular local character over one model}
(by taking $A=\emptyset$). 
\begin{proof}
This follows readily from Theorem \ref{thm:singular local character over one model}
applied to a countable reduct of $M^{Sh}$ but we elaborate.

Let $N\succ M$ be an $\left|M\right|^{+}$-saturated model containing
$A$. Let $M^{Sh}$ be the Shelah expansion of $M$ in the language
$L^{Sh}$, i.e., add predicates of the form $R_{\varphi\left(z,c\right)}\left(z\right)$
for every formula $\varphi\left(z,y\right)$ over $\emptyset$ and
$c\in N$, and interpret $R_{\varphi\left(z,c\right)}$ as $\varphi\left(M,c\right)$.
By \cite[Corollary 3.24]{Sh715,pierrebook}, $M^{Sh}$ is NIP. However,
note that $L^{Sh}$ is of size $\left|N\right|+\left|T\right|$. Let
$L^{A}$ be $\set{R_{\varphi\left(z,c\right)}}{\varphi\left(z,y\right)\in L,c\in A}$
and let $M^{A}$ be the reduct of $M^{Sh}$ to $L^{A}$ and $T^{A}=\Th\left(M^{A}\right)$.
Then $\left|T^{A}\right|=\left|T\right|$ because $A$ is countable.
The set 
\[
\set{R_{\varphi\left(x,y,c\right)}\left(x,m\right)}{\varphi\left(x,y,z\right)\in L,\varphi\left(x,m,c\right)\in p,m\in M,c\in A}
\]
 of formulas over $M^{A}$ is consistent because $p$ is finitely
satisfiable in $M$ and let $p^{A}$ be any complete type over $M^{A}$
containing it.  Applying the second part of Theorem \ref{thm:singular local character over one model}
to $M^{A}$ and $p^{A}$(which we are allowed to do since $\left|T^{A}\right|<\cof\left(\mu\right)<\mu$)
we find some $M_{0}'\prec M^{A}$ such that $p^{A}$ is finitely satisfiable
in $M_{0}'$ and $\left|M_{0}'\right|<\mu$. Let $M_{0}$ be the reduct
of $M_{0}'$ to $L$. Then $p\restriction\left(A\cup M\right)$ is
finitely satisfiable in $M_{0}$: if $\varphi\left(x,m,c\right)\in p$
where $\varphi\left(x,y,z\right)\in L$, $m\in M$ and $c\in A$ then
then $R_{\varphi\left(x,y,c\right)}\left(x,m\right)\in p^{A}$ so
for some $m'\in M_{0}$, $R_{\varphi\left(x,y,c\right)}\left(m',m\right)$
holds in $M^{Sh}$, which exactly means that $\varphi\left(m',m,c\right)$
holds in $N$ (and thus in $\C$).
\end{proof}
\begin{cor}
\label{cor:if p is f.s. in singular}If $T$ is NIP and $p\left(x\right)\in S\left(\C\right)$
is a global (finitary) type which is finitely satisfiable in a model
$M$ such that $\mu=\left|M\right|$ is singular with $\left|T\right|<\cof\left(\mu\right)$,
then $p$ is finitely satisfiable in some model $M_{0}\prec M$ such
that $\left|M_{0}\right|<\mu$. 
\end{cor}

\begin{proof}
Let $I=\sequence{a_{i}}{i<\omega}$ be a Morley sequence of $p$ over
$M$, namely $a_{i}\models p\restriction Ma_{<i}$ for all $i<\omega$.
Let $A=\set{a_{i}}{i<\omega}$ and let $q'=p\restriction\left(A\cup M\right)$.
By Lemma \ref{lem:countable set outside of dom(p)}, there is some
$M_{0}\prec M$ of size $\left|M_{0}\right|<\mu$ such that $q'$
is finitely satisfiable in $M_{0}$. Let $q\left(x\right)\in S\left(\C\right)$
be any global extension of $q'$ which is still finitely satisfiable
in $M_{0}$. Note that both $q,p$ are $M$-invariant (in fact finitely
satisfiable in $M$), and that $I$ is a Morley sequence of $q$ over
$M$. Hence, by \cite[Proposition 2.36]{pierrebook}, $p=q$. Since
$q$ is finitely satisfiable in $M_{0}$, it follows that $p$ is
as well. 
\end{proof}
The following example shows that Corollary \ref{cor:if p is f.s. in singular}
is false in general NTP$_{2}$ or even simple theories, in a very
strong sense.

\begin{example}
\label{exa:the random graph doesn't satisfy} Let $T$ be the theory
of the random graph in the language $\left\{ R\right\} $, and let
$\mu<\left|\C\right|$ be any cardinal. Let $M$ be a model of size
$\mu$. For every $N\prec\C$ of size $\left|N\right|<\mu$, let $v_{N}\in\C$
be such that $v_{N}$ is not connected to any vertex in $N$ but connected
to all vertices in $M\backslash N$. Let $\Gamma=\set{x\mathrela Rv_{N}}{\left|N\right|<\mu}$,
so it is a global partial type. Then $\Gamma$ is finitely satisfiable
in $M$ (since for every finitely many $N$'s, there is some element
in $M$ not in their union). But $\Gamma$ is clearly not finitely
satisfiable in any $N$ of size $\left|N\right|<\mu$. Thus, the same
is true for every extension of $\Gamma$ to a global type which is
finitely satisfiable in $M$. 
\end{example}

The following example, due to Saharon Shelah\footnote{We thank Shelah for allowing us to include this example here.}
shows that also Theorem \ref{thm:singular local character over one model}
does not hold for the random graph.
\begin{example}
\label{exa:Random graph not local character for coheir}Again let
$T$ be the theory of the random graph in the language $\left\{ R\right\} $
and let $\mu$ be some uncountable cardinal. We will construct a model
$M$ of size $\mu$ such that the type $p\left(x\right)=\set{x\mathrela Ra}{a\in M}$
is not finitely satisfiable in every set $C\subseteq M$ such that
$\left|C\right|<\mu$. 

The universe of $M$ is $\set{\left(\alpha,i\right)}{\alpha<\mu,i<\left|\alpha\right|+\aleph_{0}}$,
so clearly $\left|M\right|=\mu$. We construct the edge relation $R^{M}$
by induction. More precisely, for $\beta\leq\mu$, let $M_{\beta}=\set{\left(\alpha,i\right)\in M}{\alpha<\beta}$
and define $R^{M_{\beta}}$ by induction on $\beta$. For $\beta=0$,
$M_{\beta}$ is empty so there is nothing to do. For $\beta$ limit,
let $R^{M_{\beta}}=\bigcup_{\gamma<\beta}R^{M_{\gamma}}$. For $\beta$
of the form $2\alpha+1$ define $R^{M_{\beta}}=R^{M_{2\alpha}}$,
i.e., we do not add new edges. For $\beta$ of the form $2\alpha+2$,
consider the set of pairs $\left(A,B\right)$ where $A,B$ are finite,
disjoint, and contained in $M_{2\alpha+1}$. The number of such pairs
is $\left|M_{2\alpha+1}\right|+\aleph_{0}=\left|\alpha\right|+\aleph_{0}=\left|\set i{\left(2\alpha+1,i\right)\in M_{\beta}}\right|$,
so for each such pair we can choose $i<\left|2\alpha+1\right|+\aleph_{0}$
such that $\left(2\alpha+1,i\right)$ is connected to all vertices
in $A$ and disconnected from all vertices in $B$. This completes
the construction. 

By the construction in the even steps, $M=M_{\mu}$ is a random graph.
Note that for each $\left(\alpha,i\right)$ the number of vertices
connected to it of the form $\left(\beta,j\right)$ for $\beta<\alpha$
is finite. Note also that ({*}) for $\alpha,\beta<\mu$ and $j<\left|\beta\right|+\aleph_{0}$,
$\set{i<\left|2\alpha\right|+\aleph_{0}}{\left(\beta,j\right)\mathrela R\left(2\alpha,i\right)}$
is finite.  Let $C\subseteq M$ be of size $\left|C\right|<\mu$.
For $\alpha<\mu$, let $D_{\alpha}$ be all vertices of the form $\left(\alpha,i\right)$
for $i<\left|\alpha\right|+\aleph_{0}$. Fix some $\alpha>\left|C\right|$.
By ({*}), each vertex in $C$ is connected to at most finitely many
vertices in $D_{2\alpha}$, and as $\left|D_{2\alpha}\right|=\left|\alpha\right|+\aleph_{0}>\left|C\right|$
there is some vertex $v\in D_{2\alpha}$ not connected to any element
from $C$. Thus, the formula $x\mathrela Rv$ is not satisfied in
$C$ and we are done. 
\end{example}

We end with some questions.

A positive answer to the following question seems plausible:
\begin{question}
\label{que:non-forking singular local char}Assume $T$ is NTP$_{2}$
and that $p\left(x\right)\in S\left(\C\right)$ is a global (finitary)
type which does not fork over a model $M$ such that $\mu=\left|M\right|$
is singular with $\left|T\right|<\cof\left(\mu\right)$. Does it follow
that $p$ does not fork over some model $M_{0}\prec M$ such that
$\left|M_{0}\right|<\mu$, or at least over some $M_{0}\prec\C$ of
size $<\mu$? 

A weaker version is the same but assuming that $p$ is finitely satisfiable
in $M$.
\end{question}

Note that the first, stronger version of Question \ref{que:non-forking singular local char}
is not known even in NIP, so we ask:
\begin{question}
\label{prob:global type singular local char}Assume $T$ is NIP and
that $p\left(x\right)\in S\left(\C\right)$ is a global (finitary)
type which is invariant over a model $M$ (equivalently: non-forking
over $M$) such that $\mu=\left|M\right|$ is singular with $\left|T\right|<\cof\left(\mu\right)$.
Does it follow that $p$ is invariant over some model $M_{0}\prec M$
such that $\left|M_{0}\right|<\mu$, or at least that $p$ is invariant
over some $M_{0}\prec\C$ of size $<\mu$?
\end{question}

\bibliographystyle{alpha}
\bibliography{common2}

\end{document}